\documentclass[12pt,english,a4paper]{article}
\usepackage{babel}
\usepackage{epsfig,graphicx,times}
\usepackage{amssymb,amsmath}
\usepackage{booktabs}
\usepackage{geometry}
\usepackage{color}
\newtheorem{theorem}{Theorem}
\newtheorem{corollary}{Corollary}
\newtheorem{lemma}{Lemma}
\newtheorem{example}{Example}

\newtheorem{conjecture}{Conjecture}
\newcommand{\eop}{\hfill{$\Box$}}
\newenvironment{proof}
{\begin{trivlist}\item[]{{\sc Proof.}}}{\eop\noindent\end{trivlist}}

\newcommand{\diam}{\mbox{diam}}
\newcommand{\dist}{\mbox{dist}}
\newcommand{\conv}{\mbox{conv}}
\newcommand{\highlight}[1]{#1}
\newcommand{\highlightt}[1]{#1}
\newcommand{\highlighttt}[1]{#1}

\begin{document}
  \title{\highlight{Open sets} avoiding integral distances}
  \author{
    \begin{minipage}[t]{60mm}
     {\small \sc Sascha Kurz}\\%\thanks{The authors thank Tobias Kreisel for producing several graphics.}\\ 
     {\footnotesize Department of Mathematics}\\
     {\footnotesize University of Bayreuth}\\ 
     {\footnotesize D-95440 Bayreuth, Germany}\\
     {\footnotesize sascha.kurz@uni-bayreuth.de}
    \end{minipage}
    \begin{minipage}[t]{60mm}
    {\small \sc Valery Mishkin}\\%\thanks{The authors thank Peter Biryukov and Juris Steprans for valuable comments and discussions.}\\
    {\footnotesize Department of Mathematics and Statistics}\\
    {\footnotesize York University}\\
    {\footnotesize Toronto ON, M3J1P3 Canada}\\
    {\footnotesize vmichkin@mathstat.yorku.ca}
    \end{minipage}
  }
  \maketitle
  
  \vspace*{-8mm}
  
  \small{
  \noindent
  $$
  \begin{array}{ll}
    \text{\textbf{Keywords:}} & \text{excluded distances, Euclidean Ramsey theory, integral distances, Erd\H{o}s-type problems} \\
    \text{\textbf{MSC:}} & \text{52C10, 52A40, 51K99} \\
  \end{array}
  $$
  }
  \noindent
  \rule{\textwidth}{0.3 mm}
  \begin{abstract}
    \noindent
    We study open point sets in Euclidean spaces ${\mathbb R}^d$  without a pair of points an integral distance apart. By a
    result of Furstenberg, Katznelson, and Weiss such sets must be of Lebesgue upper density zero. We are interested in how
    large such sets can be in $d$-dimensional volume. We determine the exact values for the \highlighttt{maximum}
    volumes of the sets in terms of the number of their connected components and dimension. Here techniques from diophantine
    approximation, algebra and \highlighttt{the theory of} convex bodies come into play. Our problem can be viewed as a
    \highlighttt{counterpart} to known problems on sets with pairwise rational or integral distances. This reveals interesting
    links between discrete geometry, topology, and measure theory.
  \end{abstract}
  \noindent
  \rule{\textwidth}{0.3 mm}

\section{Introduction}
\noindent
Is there a dense set~$S$ in the plane so that all pairwise Euclidean distances between the points are rational?  This famous
open problem was posed by Ulam in 1945, see e.g.\ \cite{0668.01033,0599.01015,0086.24101}. \highlighttt{Unlike this, a} construction
of a countable dense set in the plane avoiding rational distances is not \highlighttt{hard} to find, see e.g.\
\cite[Problem 13.4, 13.9]{ProblemsSetTheory}. If all pairwise distances between \highlighttt{the} points in $S$
are integral and $S$ is non-collinear, i.e.\ not all points are located on a line, then $S$ is finite \cite{ErdoesAnning1,ErdoesAnning2}.
Having heard of this result, Ulam guessed that the answer to his question would be in the negative. Of course the rational
numbers form a dense subset of a coordinate line with pairwise rational distances; also, on a circle there are dense sets with
pairwise rational distances, see e.g.\ \cite{dense_subset,ErdoesAnning1}. It was proved by Solymosi and De Zeeuw \cite{1209.52009}
that the line and the circle are the only two irreducible algebraic curves containing infinite subsets of points with pairwise
rational distances. Point sets with rational coordinates on spheres have been considered in
\cite{springerlink:10.2478/s11533-008-0038-4}. 
There is interest in a general construction of a planar point set $S(n,k)$ of size $n$ with pairwise integral distances such
that $S(n,k)=A\cup B$ where $A$ is collinear, $|A|=n-k$, $|B|=k$, and $B$ has no three collinear points. The current record
is $k=4$ \cite{four_points_beside}.
And indeed, it is very hard to construct a planar point set, no three points on a line, no four points on a circle, with
pairwise integral distances. Kreisel and Kurz~\cite{1145.52010} found such a set of size~$7$, but it is unknown if there
exists one of size~$8$.

The present paper is concerned with a problem that may be considered as \highlighttt{a counterpart} to those just described, namely with
\textit{large} point sets in $\mathbb{R}^d$ without a pair of points an integral distance apart. We write $f_d(n)$ for the
supremum of the volumes $\lambda_d(\mathcal{P})$ of open point sets $\mathcal{P}\subset\mathbb{R}^d$ with $n$ connected
components without a pair of points whose distance apart is a positive integer. We determine the exact \highlighttt{values} of
\highlighttt{the function} $f_d(n)$ \highlighttt{for all $d$ and $n$}.

This problem is related to the famous Hadwiger--Nelson open problem of determining the (measurable) chromatic number of
$\mathbb{R}^{d}$, see e.g.\ \cite[Problem G10]{UPIG}. Here one can also ask for the highest density of one color class in
such a coloring, that is, we may ask for the densest set without a pair of points a distance $1$ apart. In \cite{1146.05025}
such a construction in $\mathbb {R}^3$ has been given.
In the plane the best known example, due to Croft~\cite{croft}, consists of the intersections of hexagons with circles and
attains a density of $0.2294$. The upper bounds are computed in \cite{10.1007s00039-009-0013-7,1205.90196}. Point sets
avoiding a finite number $k$ of prescribed distances are considered e.g.\ in \cite{1169.52005} and \cite[Problem G4]{UPIG},
so the point sets avoiding all distances that are positive integers correspond to the case with an infinite number $\kappa$
of excluded distances. It is known \cite{0738.28013} that for each subset $\mathcal{U}$ of the plane with \highlighttt{positive
density}, there is a constant $d(\mathcal{U})$ such that all distances greater than $d(\mathcal{U})$ occur \highlighttt{between the points
of $\mathcal{U}$}. The same result is true in higher dimensions \cite{FCD_higher_dimension}.
It follows that in every dimension $d\ge 2$, the Lebesgue measurable sets avoiding integral distances, which are of interest
here, must be of upper density zero, so we consider the supremum of their volumes instead.

The paper is organized as follows: in Section~\ref{sec_observations} we introduce the basic notation and
provide characterizations of \highlighttt{arbitrary open} point sets without pairs of points an integral distance apart.
After \highlighttt{stating} first \highlighttt{relationships between the} upper bounds \highlighttt{for} the maximum 
\highlighttt{volumes} of those sets \highlighttt{with different numbers of connected components} we continue in
Section~\ref{sec_collection_of_balls} by considering a relaxed problem. \highlighttt{We evaluate the maximum volumes of sets
avoiding integral distances in the special case where the connected components of the sets are open balls. In our crucial
constructions we make use of Weyl's theorem from diophantine approximation and the fact, we derived from Mann's theorem, 
that the lengths} of the diagonals of a regular $p$-gon are linearly independent over $\mathbb{Q}$ \highlighttt{whenever} $p$
is a prime. In Section~\ref{sec_bounds} we \highlighttt{approach} the \highlighttt{main problem of evaluating the function
$f_d(n)$ in the general case}. For \highlighttt{two-component opens sets ($n=2$) we provide a complete solution in} 
Subsection~\ref{subsec_two_components}. Motivated by \highlighttt{the} necessary conditions for \highlighttt{open} point sets
\highlighttt{to avoid} integral distances we consider $d$-dimensional open sets with $n$ connected components of diameter at
most \highlighttt{$1$} each whose intersection with every line \highlighttt{has a} total \highlighttt{length} of at most \highlighttt{$1$}.
\highlighttt{At the end of Subsection~\ref{subsec_two_components} we state a conjecture on the exact values of $l_d(n)$ that bears a strong
resemblance to the problems} of geometric tomography, see e.g.\ \cite{MR2251886}. \highlighttt{In 
Subsection~\ref{subsec_bounds_line_restricted} we provide some upper bounds for $l_d(2)$ with $d\ge 2$.}
\highlighttt{The main problem of  evaluating $f_d(n)$ generally is finally settled in Subsection~\ref{subsec_bounds_anti_integral}.}
In Section~\ref{sec_conclusion} we give a summary of the results obtained and draw the appropriate conclusions.

\section{General observations and basic notation}
\label{sec_observations}

Denote by $\dist(x,y)$  the Euclidean distance between two points $x,y\in {\mathbb R}^d$ and by
$\dist(V,W):=\inf \{\dist(x,y)\mid x\in V,y\in W\}$ the distance between two subsets $V$ and $W$ of ${\mathbb R}^d$. The
minimum width of $V$, i.e.\ the minimum distance between parallel support hyperplanes of the closed convex hull of $V$, will
be denoted by $\mbox{width}(V)$, and $\lambda_d$ will stand for the Lebesgue measure in $\mathbb{R}^d$.

At first we observe that the diameter of any connected component of an open set avoiding integral distances, i.e.\ having
no points an integral distance apart, is at most \highlighttt{$1$}.

\begin{lemma}
  \label{lemma_diameter}
  Let $\mathcal{P}\subseteq\mathbb{R}^d$ be an open set avoiding integral distances. Then for every connected component
  $\mathcal{C}$ of $\mathcal{P}$ we have $\diam(\mathcal{C})\le 1$.
\end{lemma}
\begin{proof}
  Suppose there is a connected component $\mathcal{C}$ with $\diam(\mathcal{C})> 1$, then there exist 
  $x_1,x_2\in \mathcal{C}$ such that $\dist(x_1,x_2)>1$. Since ${\mathbb R}^d$ is locally connected, $\mathcal{C}$ is open, so it is path connected. Hence there
  is a point $x$ on the image curve of a continuous path  in $\mathcal{C}$ joining $x_1$ and $x_2$ such that  $\dist(x_1,x)=1$.
\end{proof}

By the isodiametric inequality the open ball $B_d\subset\mathbb{R}^d$ centered at the origin with unit diameter has the
largest volume among measurable sets in ${\mathbb R}^d$ of diameter at most $1$, see e.g.\
\cite{0804.28001}, \cite[chap.~2]{0603.01017}. Thus we have
$$
  f_d(1)=\lambda_d\!\left(B_d\right)=\frac{\pi^{d/2}}{2^d\cdot\Gamma\left(\frac{d}{2}+1\right)}=
  \left\{ \begin{array}{rcl} \frac{\pi^{\frac{d}{2}}}{2^d\left(\frac{d}{2}\right)!} &\!\!\!& \text{for $d$ even},\\[3mm]
  \frac{\left(\frac{d-1}{2}\right)!\cdot\pi^{\frac{d-1}{2}}}{d!}&\!\!\!&\text{for $d$ odd}.\end{array}\right.
$$
The first few values are given by $\lambda_1\!\left(B_1\right)=1$, $\lambda_2\!\left(B_2\right)=\frac{\pi}{4}$,
$\lambda_3\!\left(B_3\right)=\frac{\pi}{6}$, and $\lambda_4\!\left(B_4\right)=\frac{\pi^2}{32}$. Note 
that the \highlighttt{volume of the} scaled ball \highlighttt{$B$} with diameter $m$ in $\mathbb{R}^d$ \highlighttt{is}
$\lambda_d(\highlighttt{B})=m^d\lambda_d\!\left(B_d\right)$.

Next we characterize $1$-dimensional open sets containing a pair of points an integral distance apart.

\begin{lemma}
  \label{lemma_dimension_1}
  A non-empty open set $\mathcal{P}\subseteq \mathbb R$ contains a pair of points $x,y\in \mathcal{P}$ with
  $\dist(x,y)\in\mathbb N$ if and only if either $\lambda_1(\mathcal{P})>1$ or there is a pair of connected components
  (i.e.\ disjoint open intervals) $\mathcal{C}_1$, $\mathcal{C}_2$ of $\mathcal{P}$ such that
  $\dist(\mathcal{C}_1,\mathcal{C}_2)\notin\mathbb N$ and $\lambda_1(\mathcal{C}_1\cup \mathcal{C}_2)>
  \lceil \dist(\mathcal{C}_1,\mathcal{C}_2)\rceil -\dist(\mathcal{C}_1,\mathcal{C}_2)$. If $\lambda_1(\mathcal{P})\leq 1$,
  then there exists a shift $f: x\mapsto x+a$ of $\mathbb R$ such that $f\left(\mathcal{P}\right)\cap {\mathbb Z}=\varnothing$.
\end{lemma}
\begin{proof}
The restriction of the canonical epimorphism $\phi: {\mathbb{R}}\to {\mathbb{R}}/{\mathbb{Z}}$, 
$x\mapsto x+\mathbb Z=(x-\lfloor x\rfloor) +\mathbb Z$, to the interval $[0,1)$ is a continuous bijection of $[0,1)$ onto
the $1$-dimensional torus ${\mathbb T}={\mathbb R}/{\mathbb Z}$\highlight{,} the inverse map $\phi|_{[0,1)}^{-1}$ being continuous at all
points except $\phi(0)=0+\mathbb Z=\mathbb Z\in\mathbb T$. We consider the retraction $\phi_1:=\phi|_{[0,1)}^{-1}\circ\phi:
\mathbb R\to [0,1)$, that is, $\phi_1(x)=x-\lfloor x\rfloor$ for all $x\in\mathbb R$ (i.e.\ $\phi_1(x)=x\,\mbox{mod}\, 1$ is
the fractional part of $x$).
We observe that the image under $\phi_1$ of any open interval $(x,y)$ of length $y-x<1$ is either the open interval
$(\phi_1(x),\phi_1(y))=(x-n,y-n)$ of the same length $\phi_1(y)-\phi_1(x)=(x-n)-(y-n)=y-x$, whenever both $x$ and  $y$ are
in $(n,n+1)$, for some $n\in\mathbb Z$, or the union of two disjoint connected components
$$[0,\phi_1(y))\cup (\phi_1(x),1)=[0,y-n)\cup (1-(n-x),1)$$  of the same total length $(y-n)+(n-x)=y-x$, whenever
$x<n<y$, for some  $n\in\mathbb Z$. If $y-x=1$, then similarly either $\phi_1((n,n+1))=(0,1)$  or
$\phi_1((x,y))=[0,y-n)\cup (1-(n-x),1)=[0,1)\setminus\{y-n\}$ whenever $x<n<y$ for some $n\in\mathbb N$. Hence, in general,
the total length of the connected components of $\phi_1((x,y))$ is $y-x$, whenever $y-x\leq 1$.

Let $\mathcal{P}$ be the disjoint union of open intervals $\mathcal{C}_i$, say,  with total length
$\lambda_1(\mathcal{P})=\sum_i\lambda_1(\mathcal{C}_i)>1$. Then by Lemma~\ref{lemma_diameter} $i\geq 2$ and
$\lambda_1(\mathcal{C}_i)\le 1$ for all $i$. We thus have from above that the total length of the
connected components of all the images $\phi_1(\mathcal{C}_i)$ equals $\sum_i\lambda_1(\mathcal{C}_i)>1$. Hence at least
two images $\phi_1(\mathcal{C}_k)$ and $\phi_1(\mathcal{C}_j)$ must overlap, so there exists 
$z\in \phi_1(\mathcal{C}_k)\cap \phi_1(\mathcal{C}_j)$, that is, $x_0-\lfloor x_0\rfloor=y_0-\lfloor y_0\rfloor$ for some
$x_0\in \mathcal{C}_k$ and $y_0\in \mathcal{C}_j$. Thus $x_0-y_0=\lfloor x_0\rfloor-\lfloor y_0\rfloor\in\mathbb Z\setminus\{0\}$,
hence $\dist(x_0,y_0)\in\mathbb N$.

If $\lambda_1(\mathcal{C}_1\cup \mathcal{C}_2)>\alpha$ for some connected components $\mathcal{C}_1=(a,b)$ and
$\mathcal{C}_2=(c,d)$ of $\mathcal{P}$ with $\dist(\mathcal{C}_1,\mathcal{C}_2)=c-b=m-\alpha$, where $m\in\mathbb N$,
$0<\alpha<1$, so that $\lceil \dist(\mathcal{C}_1,\mathcal{C}_2)\rceil - \dist(\mathcal{C}_1,\mathcal{C}_2)=\alpha$, we
can take a point $x$ in the leftmost interval, say $x\in \mathcal{C}_1$ and a point $y\in \mathcal{C}_2$ so that the
length of $(x,b)\cup (c,y)$ is $\alpha<\lambda_1(\mathcal{C}_1\cup \mathcal{C}_2)=(b-a)+(d-c)$. Then
$$\dist(x,y)=(b-x)+m-\alpha+(y-c)=\alpha+m-\alpha=m\in\mathbb N.$$

Conversely, suppose there are $x,y\in \mathcal{P}$ with $\dist(x,y)=k\in\mathbb N$. If $x$ and $y$ lie in the same connected
component $\mathcal {C}_i$ of $\mathcal{P}$, then $\lambda_1(\mathcal{C}_i)>k\geq 1$ because $\mathcal{C}_i$ is open, hence
$\lambda_1(\mathcal{P})>1$. \highlight{Suppose} $x$ and $y$ lie in distinct connected components of $\mathcal{P}$, say $x<y$ and
$x\in \mathcal{C}_1=(a,b)$, $y\in \mathcal{C}_2=(c,d)$, and let $\lambda_1(\mathcal{P})\leq 1$. Then $(b-a)+(d-c)\leq 1$ as
well whence the distance between the components is $\dist(\mathcal{C}_1,\mathcal{C}_2)=c-b\notin\mathbb N$, because
$c-b<\dist(x,y)<c-b+[(b-a)+(d-c)]\highlight{\le}c-b+1$. Let $c-b=m-\alpha$ where $m\in\mathbb N,$ $0<\alpha<1$. Then 
$$\alpha=m+b-c<m+1+b-c<(d-a)+(b-c)=(b-a)+(d-c)=\lambda_1(\mathcal{C}_1\cup\mathcal{C}_2),$$ since $m+1<d-a$ because
$m+1\leq k<d-a$. Thus either $\lambda_1(\mathcal{P})>1$ or there is a pair of required connected components of $\mathcal{P}$.

If $\lambda_1(\mathcal{P})\leq 1$, then $\lambda_1(\mathcal{C}_i)\leq 1$ for all $i,$ so the total length of the connected
components of all the images $\phi_1(\mathcal{C}_i)$ equals $\sum_i\lambda_1(\mathcal{C}_i)=\lambda_1(\mathcal{P})$, as 
\highlight{shown previously}.  If $\lambda_1(\mathcal{P})< 1$, then clearly, $\phi_1(\mathcal{P})\neq [0,1)$. If $\lambda_1(\mathcal{P})=1$,
then again $\phi_1(\mathcal{P})\neq [0,1)$, whenever the images ${\phi_1(\mathcal C}_i)$ are not pairwise disjoint. Suppose
all the images  ${\phi_1(\mathcal C}_i)$ are pairwise disjoint and $\mathcal{P}\cap {\mathbb Z}\neq\varnothing$. Then there
is exactly one $\mathcal{C}_j=(a,b)$ that meets $\mathbb Z$. Hence the complement $[0,1)\setminus
\phi_1(\mathcal{C}_j)=[\phi_1(b),\phi_1(a)]$ is a non-open set in $\mathbb R$ that can not be covered by the images
$\phi_1(\mathcal{C}_i)$ of the other connected components of $\mathcal{P}$, since they are all open intervals, so
$\phi_1(\mathcal{P})\neq [0,1)$ as well. Thus in all the cases we have $\phi_1(\mathcal{P})\neq [0,1)$. Take
$\phi_1(a)\in[0,1)\setminus \phi_1(\mathcal{P})$, $a\in\mathbb R$, that is, $\phi_1(a)\cap\phi_1(\mathcal{P})=\varnothing$.
Then $\phi(a)\cap\phi(\mathcal{P})=\varnothing$, i.e.\ $(a+{\mathbb Z})\cap (\mathcal{P}+{\mathbb Z})=\varnothing,$ so
$(\mathcal{P}-a)\cap\mathbb Z=\varnothing$ and the required shift is $f:x\mapsto x+(-a)$.
\end{proof}

Applying Lemma~\ref{lemma_dimension_1} we establish a criterion for \highlighttt{an open} set to avoid integral distances in all dimensions.
\begin{theorem}
  \label{thm_line_intersection}
  An open point set $\mathcal{P}\subseteq \mathbb{R}^d$ does not contain a pair of points an integral distance apart if and only if for every
  line $\mathcal{L}$
  \begin{itemize}
   \item[(i)]  $\lambda_1(\mathcal{P}\cap \mathcal{L})\leq 1$ and 
   \item[(ii)] if $\mathcal{L}$ hits a pair of distinct connected components $\mathcal{C}_1$, $\mathcal{C}_2$ of $\mathcal{P}$
               in the intervals $\mathcal{C}_1\cap \mathcal L,\mathcal{C}_2\cap \mathcal{L}$ with $\dist(\mathcal{C}_1\cap \mathcal{L},
               \mathcal{C}_2\cap \mathcal{L})=r\notin \mathbb{N}$, then
               $\lceil r\rceil -r\ge \lambda_1((\mathcal{C}_1\cup \mathcal{C}_2)\cap \mathcal{L})$.
  \end{itemize}
\end{theorem}

Another criterion, which we will also be using is:

\begin{lemma}
  \label{lemma_prestage_width}
  Let $\mathcal{P}$ be a $d$-dimensional \highlighttt{disconnected} open set \highlighttt{all of whose connected} components \highlighttt{are of}
  diameter at most $1$. \highlight{Then} $\mathcal{P}$ contains a pair of points with integral distance if and only if
  $$\Bigl(\dist(\mathcal{C}_1,\mathcal{C}_2),\diam(\mathcal{C}_1\cup\mathcal{C}_2)\Bigr)\cap\mathbb{N}\neq\emptyset$$ for \highlighttt{some} of its
  \highlighttt{connected} components $\mathcal{C}_1,\mathcal{C}_2$.
\end{lemma}
\begin{proof}
  Since all the \highlighttt{connected} components \highlighttt{of $\mathcal{P}$} are open \highlighttt{with} diameter at most $1$, 
  \highlighttt{any two distinct} points \highlighttt{of $\mathcal{P}$} with integral distance \highlighttt{must be in} two different components,
  say $\mathcal{C}_1$ and $\mathcal{C}_2$. Let $x\in\mathcal{C}_1$, $y\in\mathcal{C}_2$ with $\dist(x,y)=n\in\mathbb{N}$. We then select two small
  closed balls $\overline{B}(x,\varepsilon_1)\subsetneq\mathcal{C}_1$ and $\overline{B}(y,\varepsilon_2)\subsetneq\mathcal{C}_2$ centered at $x$ and $y$
  respectively \highlighttt{with radii} $\varepsilon_1,\varepsilon_2>0$. The line $\mathcal{L}$ through $x$ and $y$ meets the two balls in 
  \highlighttt{the} intervals, \highlighttt{say} $[x_1,x_2]\subsetneq \highlight{\overline{B}}(x,\varepsilon_1)$
  and $[y_1,y_2]\subsetneq \highlight{\overline{B}}(y,\varepsilon_2)$, where $x_1,x_2\in\mathcal{C}_1$ and $y_1,y_2\in\mathcal{C}_2$. With this notation we have
  $$
  \dist(\mathcal{C}_1,\mathcal{C}_2)<\min_{1\le i,j\le 2}\dist(x_i,y_j)<\dist(x,y)=n<\max_{1\le i,j\le 2}\dist(x_i,y_j)<
  \diam(\mathcal{C}_1\cup\mathcal{C}_2). 
  $$
  Conversely, if $\dist(\mathcal{C}_1,\mathcal{C}_2)<n<\diam(\mathcal{C}_1\cup\mathcal{C}_2)$ for an integer $n$, then there exist $x_1,x_2\in\mathcal{C}_1$ and
  $y_1,y_2\in\mathcal{C}_2$ such that
  $$
  \dist(\mathcal{C}_1,\mathcal{C}_2)<\dist(x_1,y_1)<n<\dist(x_2,y_2)<\diam(\mathcal{C}_1\cup\mathcal{C}_2). 
  $$
  Joining $x_1$ with $x_2$ in $\mathcal{C}_1$ and $y_1$ with $y_2$ in $\mathcal{C}_2$ by continuous paths, we can find $x\in\mathcal{C}_1$ and $y\in\mathcal{C}_2$
  on \highlighttt{the image curves of} these paths with $\dist(x,y)=n$.
\end{proof}

\highlight{Sometimes it is helpful, if we can assume that the connected components of the point sets in question are not
too close to each other. Specifically, we will be using the fact that in such cases the connected components of the sets have 
disjoint closures.} 

\begin{lemma}
  \label{lemma_repultion_property}
  Let $\mathcal{C}_1$, $\mathcal{C}_2$ be distinct connected components of a $d$-dimen\-sio\-nal open point set $\mathcal{P}$
  without a pair of points an integral distance apart. If $\lambda_d(\mathcal{C}_1\cup \mathcal{C}_2)>\lambda_d\!\left(B_d\right)$,
  then $\dist(\mathcal{C}_1,\mathcal{C}_2)\ge 1$.
\end{lemma}
\begin{proof}
  \highlight{Making use of} the isodiametric inequality \highlight{we deduce from $\lambda_d(\mathcal{C}_1\cup \mathcal{C}_2)>\lambda_d\!\left(B_d\right)$
  that $\diam(\mathcal{C}_1\cup \mathcal{C}_2)>1$. By} 
  Lemma~\ref{lemma_diameter} we have $\diam(\mathcal{C}_1)\le 1$ and $\diam(\mathcal{C}_2)\le 1$. So we can choose
  $x_1\in \mathcal{C}_1$, $x_2\in \mathcal{C}_2$ with $\dist(x_1,x_2)>1$. If $\dist(\mathcal{C}_1,\mathcal{C}_2)<1$, then there
  exist $\bar x_1\in \mathcal{C}_1$ and $\bar x_2\in \mathcal{C}_2$ such that $\dist(\bar x_1,\bar x_2)<1$. Since $\mathcal{C}_1$
  and $\mathcal{C}_2$ are open, they are path connected, hence we can join $x_1$ and $\bar x_1$ by a continuous path in
  $\mathcal{C}_1$ and similarly $x_2$ and $\bar x_2$ in $\mathcal{C}_2$ and on the image curves of these paths we then find
  $x_1'\in \mathcal{C}_1$ and $x_2'\in \mathcal{C}_2$ such that $\dist(x_1',x_2')=1$, but $\mathcal{P}$ avoids integral distances,
  a contradiction. Thus we have $\dist(\mathcal{C}_1,\mathcal{C}_2)\geq 1$.
\end{proof}

As Lemma~\ref{lemma_diameter} and Theorem~\ref{thm_line_intersection}(i) will be our main tools in estimating upper bounds
for $f_d(n)$, we denote by $l_d(n)$ the supremum of the volumes $\lambda_d(\mathcal{P})$ of open point sets
$\mathcal{P}\subseteq\mathbb{R}^d$ \highlighttt{with} $n$ connected components \highlighttt{of diameter} at most~$1$ \highlighttt{each}
(condition (a)), \highlighttt{and with} total length of the intersection with every line at most $1$ (condition (b)). Clearly
$l_d(1)=f_{\highlight{d}}(1)=\lambda_d\!\left(B_d\right)$ and $f_d(n)\le l_d(n)$ \highlighttt{for all $d$ and $n$}. Note that
omitting condition (b) \highlighttt{trivializes the problem of estimating the extreme volumes, the} extreme configurations 
\highlighttt{obviously} consist of $n$ disjoint open $d$-dimensional balls of diameter $1$. Dropping condition (a) makes the problem more
challenging. It turns out that there are open connected $d$-dimensional point sets $\mathcal{P}$ with infinite volume $\lambda_d(\mathcal{P})$
and diameter $\diam(\mathcal{P})$ even though the length of the intersection of $\mathcal{P}$ with every line $\mathcal{L}$ is at
most $1$, i.e.\ $\lambda_1(\mathcal{P}\cap\mathcal{L})\le 1$ \highlighttt{for all $\mathcal{L}$}.

\begin{example}
  \label{annuli_construction}
  For integers $n\ge 1$ and $d\ge 2,$ denote by $\mathcal{A}_n^d$  the $d$-dimensional open spherical shell, or annulus,
  centered at the origin with inner radius $n$ and outer radius $n+\frac{1}{dn^d}$, i.e.\ $\mathcal{A}_n^d$ are bounded by
  concentric $(d-1)$-dimensional spheres centered at the origin. These shells will guarantee that the
  volume of their union is unbounded as $n$ increases. So far the constructed point set is disconnected. To obtain
  a connected point set, we denote by $\mathcal{B}_n^d$ the $d$-dimensional
  open spherical shell centered on the $y$-axis at $n+\frac{3}{4}$ with inner radius $1$ and outer radius $1+\frac{1}{n^4}$.
  With this $\mathcal{P}=\cup_{n\ge 30} \left(\mathcal{A}_n^d\cup\mathcal{B}_n^d\right)$ is open and connected with
  infinite volume and diameter even though the length of its intersection with every line is smaller than 1. In
  Figure~\ref{fig_annuli_construction} we depicted \highlighttt{such} a configuration \highlighttt{in} dimension $d=2$
  with first few annuli getting thinner and thinner 
  $\mathcal{A}_n^2$ and $\mathcal{B}_n^2$ being blue and green respectively. The detailed computations demonstrating the
  assertions claimed are provided in the \highlight{Appendix, see Subsection~\ref{subsec_annuli}}.
\end{example}

\begin{figure}[htp]
  \begin{center}
    \includegraphics{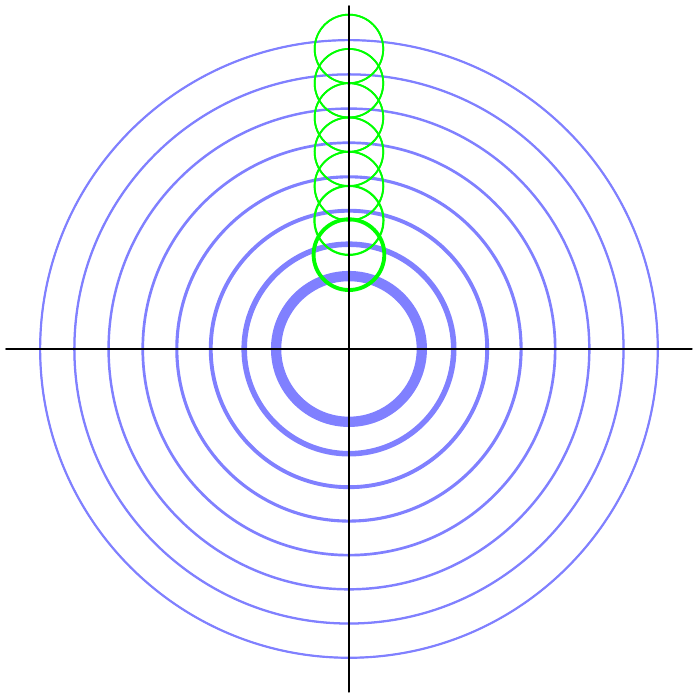}
    \caption{Concentric annuli with infinite area but small lengths of line intersections.}
    \label{fig_annuli_construction}
  \end{center}
\end{figure}

\highlighttt{In order to make the problem of evaluating the functions $f_d(n)$ 
and $l_d(n)$ more tractable}, we consider both problems in the special case, where the connected components are restricted to $d$-dimensional
open balls. We denote the corresponding maximum volumes by $f^\circ_d(n)$ and $l^\circ_d(n)$ respectively. Clearly we have $f^\circ_d(n)\le f_d(n)$ and $f^\circ_d(n)\le l^\circ_d(n)\le l_d(n)$. In Section~\ref{sec_collection_of_balls} we determine the \highlighttt{exact} values of both functions $l^\circ_d(n)$
and $f^\circ_d(n)$ \highlighttt{for all $d$ and $n$}.

Based on a simple averaging argument, any given upper bound on one of the four introduced maximum volumes for $n$ connected components
\highlighttt{yields} an upper bound for $k\ge n$ connected components in the same dimension.
\begin{lemma}
  \label{lemma_averaging}
  For each $k\ge n$ we have \highlightt{$l_d(k)\le \frac{k}{n}\cdot \Lambda_1$ whenever $l_d(n)\le \Lambda_1$ and 
  $f_d(k)\le \frac{k}{n}\cdot \Lambda_2$ whenever $f_d(n)\le \Lambda_2$}.
\end{lemma}
\begin{proof}
  Let $\mathcal{P}$ be a $d$-dimensional open set with \highlighttt{corresponding property in either case and $k\ge n$} connected components.
  The volume of each of the ${k\choose n}$ different \highlighttt{unions} of $n$ connected components \highlighttt{inheriting these properties}
  is at most $\Lambda_i$. Since each connected component occurs exactly ${{k-1}\choose{n-1}}$ times in those \highlighttt{unions,} the stated
  \highlighttt{inequalities hold}.
\end{proof}

\section{\highlighttt{Unions} of open $\mathbf{d}$-dimensional balls}
\label{sec_collection_of_balls}
Here we consider open point sets \highlighttt{$\mathcal{P}$ that are unions of $n$ disjoint}  $d$-dimensional open balls of diameter at most
\highlighttt{$1$ each} such that \highlighttt{they} either do \highlighttt{not} contain a pair of points with integral distance or \highlighttt{intersect}
each line \highlighttt{in the intervals with total length} at most \highlighttt{$1$}. As introduced in the previous section, we denote the supremum of possible volumes \highlighttt{of such $\mathcal{P}$} by \highlighttt{$f_d^\circ(n)$} and \highlighttt{$l_d^\circ(n)$} respectively.

In dimension \highlighttt{$1$} we can consider one open interval of length $1-\varepsilon$ and $n-1$ open intervals of length
$\frac{\varepsilon}{n}$, where $1>\varepsilon>0$, arranged in \highlighttt{a} unit interval so that they are pairwise \highlighttt{disjoint}.
\highlighttt{Clearly the set does not have a} pair of points an integral distance apart and the total length of the $n$ intervals tends to $1$, as
$\varepsilon$ approaches \highlighttt{$0$}. \highlighttt{It follows from Theorem~\ref{thm_line_intersection} that} $f^\circ_1(n)=l^\circ_1(n)=1$
\highlighttt{for all $n$}. For $n=1$, \highlighttt{by the isodiametric inequality, only the volumes of $d$-dimensional open balls of diameter $1$
attain the maximum value} $f^\circ_d(1)=l^\circ_d(1)=\lambda_d\!\left(B_d\right)$.

\begin{lemma}
  \label{lemma_upper_bound_circle}
   $l^\circ_d(n)\le \max\!\left(1,\frac{n}{2^d}\right)\cdot\lambda_d\!\left(B_d\right)$ for all $d,n\in\mathbb{N}$.
\end{lemma}
\begin{proof}
  Consider $n$ disjoint open $d$-dimensional balls with diameters $X_1\le 1,\cdots,X_n\le 1 $, where we can assume w.l.o.g.\ that $X_1\le \dots\le X_n\le 1$.
  Clearly in dimension $1$ we have $l_1^{\circ}(n)=l_1(n)=1=\lambda_1(B_1)$ for all $n\in\mathbb{N}$, and for all dimensions $d$, we have 
  $l_d^{\circ}(1)=l_d(1)=\lambda_d(B_d)$, so in the cases where either $d$ or $n$ is $1$ the stated inequality holds.  Hence we can assume that
  $d\ge 2$ and $n\ge 2$. \highlighttt{Then} by Theorem~\ref{thm_line_intersection}(i) we have $X_i+X_j\le 1$ for all $1\le i<j\le n$.
  If $X_n\le \frac{1}{2}$, then $\sum\limits_{i=1}^n X_i^d \le \frac{n}{2^d}$, so \highlighttt{the required inequality holds}. Otherwise we have $X_i\le 1-X_n$
  and \highlighttt{it remains to maximize} the function $g_d(x):=x^d+(n-1)(1-x)^d$ \highlighttt{with domain} $\left[\frac{1}{2},1\right]$. \highlighttt{Since}
  $g_d(x)''=d(d-1)x^{d-2}+d(d-1)(n-1)(1-x)^{d-2}>0$, every inner local extremum \highlighttt{of $g_d$ is} a minimum, so the global maximum 
  \highlighttt{of $g_d$} is attained at the boundary of the domain. Finally, we compute $g_d(1)=1$, $g_d\left(\frac{1}{2}\right)=\frac{n}{2^d}$, 
  \highlighttt{so} the \highlighttt{lemma} follows.
\end{proof}

\noindent
\highlighttt{\textbf{Remark.} The} special case of balls of diameter $\frac{1}{2}$ is directly related to point sets with
pairwise integral distances. Let $\mathcal{P}$ be \highlighttt{the union of $n$} $d$-dimensional open balls of diameter $\frac{1}{2}$ 
\highlighttt{each} without a pair of points an integral distance apart. Then the distances between the centers of the
balls $i$ and $j$ in some enumeration must be of the form $d_{ij}+\frac{1}{2}$ for \highlighttt{some} integers $d_{ij}$. By dilation with a factor of two we obtain 
\highlighttt{the set of size $n$ of the centers of the balls with pairwise odd integral distances. However, it has been shown in} \cite{0277.10021}
\highlighttt{that for such sets} $n\le d+2$, where equality holds
if and only if $d+2\equiv 0\pmod {16}$. The \highlight{exact} maximum number of odd integral distances
between points in the plane \highlight{has been} determined in \cite{0858.52007}.

\begin{theorem}
  \label{thm_l_circ}
   $l_d^\circ(n)=\max\!\left(1,\frac{n}{2^d}\right)\cdot\lambda_d\!\left(B_d\right)$ for all $n$ and  $d\ge 2$.
\end{theorem}
\begin{proof}
  \highlight{By} Lemma~\ref{lemma_upper_bound_circle} it \highlighttt{suffices} to provide \highlighttt{configurations whose volumes} (asymptotically) 
  \highlighttt{attain} the upper bound.
  
  For $1>\varepsilon>0$, we consider \highlighttt{the union of} one $d$-dimensional open ball of diameter $1-\varepsilon$ and $n-1$ 
  \highlighttt{disjoint} open balls of diameter $\frac{\varepsilon}{n-1}$ arranged in the interior of an open ball of diameter $1$. 
  As $\varepsilon$ approaches \highlighttt{$0$,} the volume of the \highlighttt{union} tends to $\lambda_d\!\left(B_d\right)$.
  
  For the remaining part we consider \highlighttt{the union of} $n$ open $d$-dimensional balls with diameter $\frac{1}{2}$ \highlighttt{centered} at
  the \highlighttt{vertices} of a regular $n$-gon with \highlighttt{circumradius} $k$. \highlighttt{Clearly, for $k$} large enough,
  \highlighttt{every line hits at most two balls}.
\end{proof}

An alternative construction would be the union of $n$ open $d$-dimensional balls of diameter $\frac{1}{2}$ centered at  $\left(i\cdot k,i^2,0,\dots,0\right)$ for $1\le i\le n$.
If $k$ is large enough, then again there is no line intersecting three or more balls.

\begin{corollary}
  \label{col_small_n}
   $f_d^\circ(n)=l_d^\circ(n)=\lambda_d\!\left(B_d\right)$ for all $d\ge 2$ and $n\le 2^d$.
\end{corollary}

It turns out that, in fact, the equalities $f_d^\circ(n)=l_d^\circ(n)=\max\!\left(1,\frac{n}{2^d}\right)\cdot\lambda_d\!\left(B_d\right)$ 
hold \highlighttt{in} all dimensions $d\ge 2$. To explain the underlying idea, we first consider the special case where $d=2$ and $n=5$,
i.e.\ the first case \highlighttt{that} is not covered by Corollary~\ref{col_small_n}.

\begin{lemma}
  \label{lemma_2_5_circ}
  $$f_2^\circ(5)=\frac{5\pi}{16}\approx 0.9817477.$$
\end{lemma}
\begin{proof}
  For each integer $k\ge 2$ and $\frac{1}{7}>\varepsilon>0$, we consider a regular pentagon~$P$ with side length
  $\frac{1}{2}-2\varepsilon+k$ and the union $U$ of five open round discs of diameter $\frac{1}{2}-2\varepsilon$ centered at the \highlighttt{vertices} of $P$, 
  see Figure~\ref{fig_pentagon_circles}. Since each connected component of $U$ has \highlight{diameter} less than $1$, there is no pair of points an integral
  distance apart inside a connected component. For every two points $a$ and $b$ from different components, we either have
  $$
    k<\operatorname{dist}(a,b)<k+1,
  $$
  whenever the discs are adjacent with their centers \highlight{located} on an edge of $P$, \highlight{or}
  $$
    \left(\frac{1+\sqrt{5}}{2}\right)\cdot k+\frac{\sqrt{5}-1}{4}-2\varepsilon<\operatorname{dist}(a,b)<
    \left(\frac{1+\sqrt{5}}{2}\right)\cdot k+ \frac{3+\sqrt{5}}{4}-5\varepsilon
  $$
  otherwise.
   
  Let $[\alpha]$ stand for the positive fractional part of a real number~$\alpha$, i.e.
  $[\alpha]:=\alpha-\lfloor\alpha\rfloor$. If, given $\varepsilon>0$, \highlighttt{one} can find an integer~$k$ such that 
  $\left[\left(\frac{1+\sqrt{5}}{2}\right)\cdot k +\frac{\sqrt{5}-1}{4}-2\varepsilon\right]<3\varepsilon$,
  then the set $U$ with parameters $k$ and $\varepsilon$ does not contain a pair of points with integral distance.
  
  Since $\frac{1+\sqrt{5}}{2}$ is irrational, we can apply the equidistribution theorem, see e.g.\ \cite{1026.42001,46.0278.06},
  to \highlighttt{ensure} that $\left(\frac{1+\sqrt{5}}{2}\right)\cdot\mathbb{N}$ is dense (even uniformly distributed) in $[0,1)$. The
  same \highlighttt{holds} true if we \highlighttt{shift the set by the} fixed real number $\frac{\sqrt{5}-1}{4}-2\varepsilon>0$.  Thus
  we can find a suitable integer~$k$ for each $\varepsilon>0$. As $\varepsilon$ approaches  \highlighttt{$0$,} the total area of $U$ tends
  to  $\frac{5\pi}{16}$, which is best possible by Lemma~\ref{lemma_upper_bound_circle}.
\end{proof}

We \highlight{illustrate} this by a short list of suitable \highlighttt{values of} $k$: 
$\left[\frac{\sqrt{5}-1}{4}+\left(\frac{1+\sqrt{5}}{2}\right)\cdot \mathbf{6}\right]\approx 0.01722$,\\
$\left[\frac{\sqrt{5}-1}{4}+\left(\frac{1+\sqrt{5}}{2}\right)\cdot \mathbf{61}\right]\approx 0.00909$, 
$\left[\frac{\sqrt{5}-1}{4}+\left(\frac{1+\sqrt{5}}{2}\right)\cdot \mathbf{116}\right]\approx 0.00096$,\\
$\left[\frac{\sqrt{5}-1}{4}+\left(\frac{1+\sqrt{5}}{2}\right)\cdot \mathbf{1103}\right]\approx 0.00051$, and
$\left[\frac{\sqrt{5}-1}{4}+\left(\frac{1+\sqrt{5}}{2}\right)\cdot \mathbf{2090}\right]\approx 0.00005$.

\begin{figure}[htp]
  \begin{center}
    \includegraphics[width=9cm]{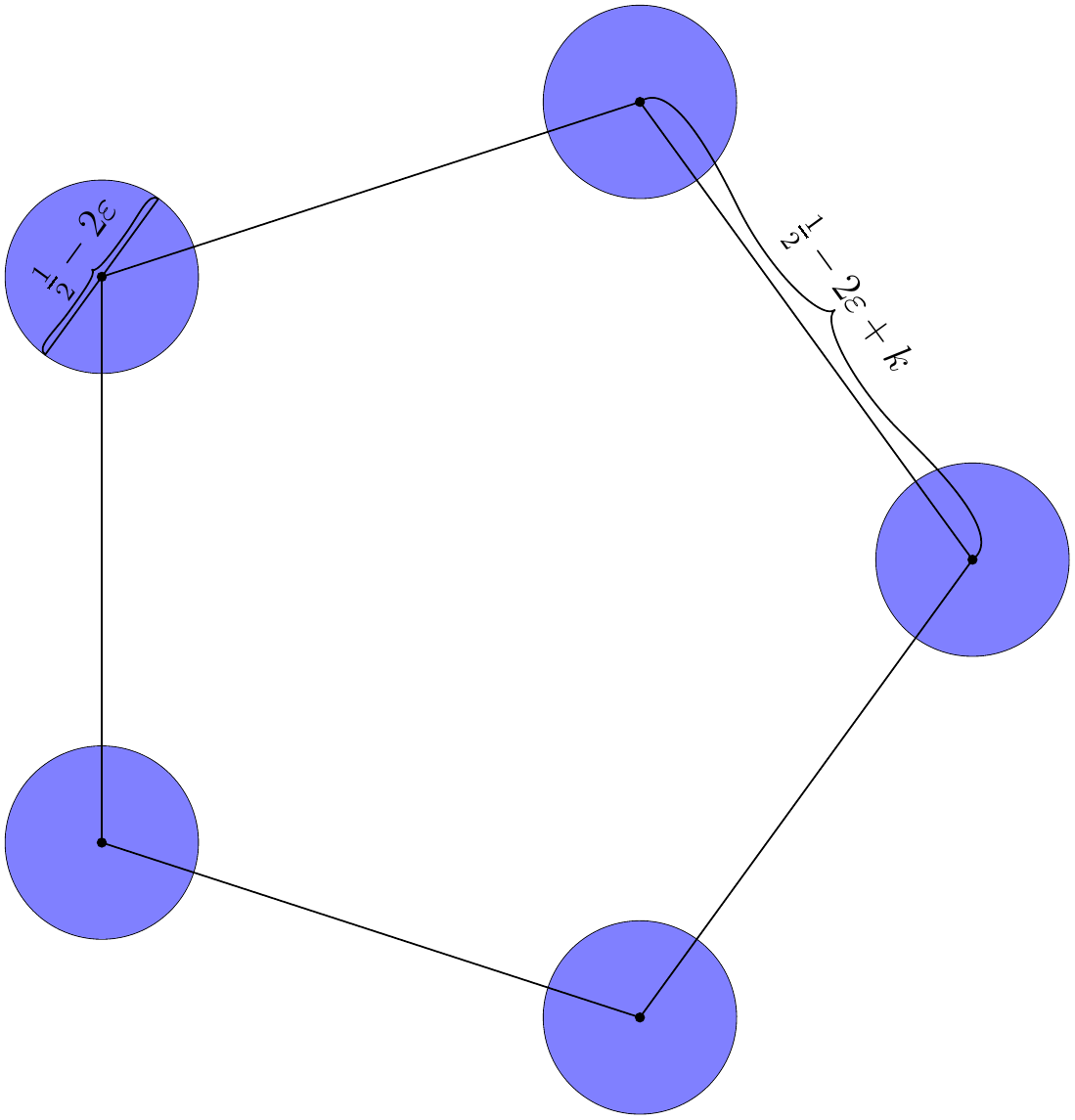}
    \caption{Integral distance avoiding \highlight{open} set for $d=2$ and $n=5$.}
    \label{fig_pentagon_circles}
  \end{center}
\end{figure}

We shall generalize Lemma~\ref{lemma_2_5_circ} to an arbitrary dimension $d\ge 2$ and arbitrary number $n$ of connected components.
The idea is to locate \highlighttt{the centers of} $n$ small $d$-dimensional open balls of diameter slightly less than $\frac{1}{2}$ at \highlighttt{some}
points $C_i$ in a two-dimensional sub-plane \highlighttt{so} that the set of different pairwise distances $\alpha_i$ between \highlighttt{their}
centers are linearly independent over the rational numbers, that is, the distances are either confluent or rationally independent. \highlighttt{The
appropriate} candidates for the center points $C_i$ \highlighttt{would be} the \highlighttt{vertices} of a regular $p$-gon, where $p$ is
an odd prime. We use a theorem of Mann, see \cite{0138.03102}, to prove \highlighttt{the desired property of} the set of distances. The condition that the
point set \highlighttt{in question} avoids integral distances can be translated into a system \highlighttt{of inequalities} of the form
$[\alpha_1\cdot k]<\varepsilon,\dots,[\alpha_l\cdot k]<\varepsilon$, where $k\in\mathbb{N}$, \highlighttt{and we are looking for} an integer $k$ such that
the \highlighttt{above} fractional \highlighttt{parts} of the scaled pairwise distances are arbitrarily small. By a theorem of Weyl, see e.g.\ \cite[Satz 3]{46.0278.06}
or a textbook on Diophantine Approximation like e.g.\ \cite{MR2953186}, such systems have solutions whenever the $\alpha_i$ are
\highlight{irrational and} linearly independent over $\mathbb{Q}$. \highlighttt{(Weyl actually proves equidistribution while we only need
denseness, a weaker result that Weyl himself attributes to Kronecker.)}

\highlight{Note that} the same construction, using the \highlighttt{vertices} of a regular hexagon, does not work. Indeed, there would be only three disinct values for the lengths $l_i$ of the diagonals, namely  $1$, $\sqrt{3}$, and $2$. The required inequalities
$$
  \left[\left(k+\frac{1}{2}-2\varepsilon\right)\cdot l_i-\left(\frac{1}{2}-2\varepsilon\right)\right]=\left[k\cdot l_i+\left(\frac{1}{2}l_i
  -\frac{1}{2}\right)+(2-2l_i)\cdot\varepsilon\right]<3\varepsilon,
$$
would trivially hold for $l_i=1$, but fail for $l_i=2$ and  $\varepsilon$ small enough.
We note in passing that quite recently Mann's theorem was used in another problem from Discrete \highlighttt{Geometry} see \cite{phd_zeeuw,rational_distances_and_angles}.

\begin{theorem} (Mann, 1965, \cite{0138.03102}) Suppose we have
 $$
   \sum_{i=1}^ka_i\zeta_i=0,
 $$
 with $a_i\in\mathbb{Q}$, $\zeta_i$ roots of unity, and no sub-relations $\sum\limits_{i\in I}a_i\zeta_i=0$, where $\emptyset\neq I\subsetneq \{1,...,k\}$. Then
 $$
   \left(\zeta_i/\zeta_j\right)^m=1
 $$
 for all $i,j$, where $m=\prod\limits_{p\text{ prime}\le k} p$.
\end{theorem}

The \highlight{vertices} of a regular $p$-gon with a circumcircle of radius $1$ centered at the origin \highlight{are given by}
$$
  \left(\cos\Bigl(\frac{j\cdot 2\pi}{p}\Bigr),\sin\Bigl(\frac{j\cdot 2\pi}{p}\Bigr)\right)
$$
for $0\le j\le p-1$. In standard complex number \highlight{notation} with $i:=\sqrt{-1}$ they coincide with the $p$th roots of unity $\zeta_j'=\cos\Bigl(\frac{j\cdot 2\pi}{p}\Bigr)+i\cdot\sin\Bigl(\frac{j\cdot 2\pi}{p}\Bigr)$.
The distance between the \highlight{vertices $0$ and $j$} is equal to $2\sin\Bigl(\frac{j\cdot 2\pi}{2p}\Bigr)$. Since
$\sin(\pi-\alpha)=\sin(\pi)$, there are only $(p-1)/2$ distinct distances in a regular $p$-gon, attained for $1\le j\le (p-1)/2$.
We note in passing that this number is not far away from the minimum number of distinct distances in the plane, which is \highlight{bounded below} by
$c\cdot\frac{p}{\log p}$ for a suitable constant $c$, see \cite{distinct_distances_plane}. We can express these distances
in terms of \highlight{$2p$th} roots of unity $\zeta_j=\cos\Bigl(\frac{j\cdot 2\pi}{2p}\Bigr)+i\cdot\sin\Bigl(\frac{j\cdot 2\pi}{2p}\Bigr)$
via
$$
  2\sin\Bigl(\frac{j\cdot 2\pi}{2p}\Bigr)=\frac{\zeta_j-\zeta_{2p-j}}{i}
$$
for all $1\le j\le \frac{p-1}{2}$.

\begin{lemma} 
  \label{lemma_diagonals_p_gon}
  \highlight{Given an odd prime $p$,} let $\alpha_j=\frac{\zeta_j-\zeta_{2p-j}}{i}$ for $1\le j\le \frac{p-1}{2}$, where the $\zeta_j$ are \highlight{$2p$th}
  roots of unity. Then the $\alpha_j$ are \highlight{irrational and} linearly independent over $\mathbb{Q}$.
\end{lemma}
\begin{proof}
  \highlight{A folklore result, see e.g.\ \cite{MR1526309}, states that $\sin(\pi q)$, where $q\in\mathbb{Q}$, is a rational number
  if and only if $\sin(\pi q)\in\left\{-1,-\frac{1}{2}, 0,\frac{1}{2},1\right\}$. Since $p$ is odd, this cannot \highlight{occur} in our context.
  \highlight{It} remains to show that the irrational numbers $\alpha_j$ are linearly independent over $\mathbb{Q}$.} Suppose to the contrary that there are
  rational numbers $b_j$ for $1\le j\le l\le\frac{p-1}{2}$ such that $\sum_{j=1}^{l} b_j\alpha_j=0$. \highlight{We then} have 
  $$\sum\limits_{j=1}^{l}\left(b_j\zeta_j-b_j\zeta_{2p-j}\right)=0.$$ Now let $J$ be a subset of those indices $j$, $2p-j$ such that $\sum_{j\in J} a_j\zeta_j=0$,
  where $a_j\in\{\pm b_j\}$, and no vanishing sub-combination. We have $|J|\le p-1$. \highlight{Hence by} Mann's Theorem $(\zeta_{j_1}/\zeta_{j_2})^2=1$ for all
  $j_1,j_2\in J$, since $$\gcd\Bigl(2p,\prod_{t\text{ prime}\le p-1} t\Bigr)=2.$$ This yields $j_2=j_1+p$ for $j_1<j_2$. Since $J$ is a subset of
  $$
    \left\{1,\dots,\frac{p-1}{2}\right\}\cup\left\{2p-\frac{p-1}{2},\dots,2p-1\right\},
  $$
  this is \highlight{impossible}, so the numbers $\alpha_j$ have to be linearly independent over $\mathbb{Q}$.
\end{proof}

\begin{theorem}
  \label{construction_f_circ_d_n}
    $f^\circ_d(n)=\max\!\left\{1,\frac{n}{2^d}\right\}\cdot\lambda_d\!\left(B_d\right)$ for all $n$ and $d\ge 2$.
\end{theorem}
\begin{proof}
  Since $f_d^\circ(n)\le l_d^\circ(n)$, by Theorem~\ref{thm_l_circ} 
  $f^\circ_d(n)\le\max\!\left\{1,\frac{n}{2^d}\right\}\cdot\lambda_d\!\left(B_d\right)$. By Corollary 1 we can assume that $n>2^d$.
  For the construction we fix an odd prime $p$ with $p\ge n$.  For each integer $k\ge 2$ and each $\frac{1}{4}>\varepsilon>0$
  we consider a regular $p$-gon $P$ with side lengths $2k\cdot\sin\Bigl(\frac{\pi}{p}\Bigr)$, \highlight{i.e.\ with circumradius $k$}.
  At $n$ arbitrarily chosen \highlighttt{vertices} of the $p$-gon $P$ we place the centers of
  $d$-dimensional open balls with diameter $\frac{1}{2}-2\varepsilon$ and consider their union. Since each of the $n$ connected
  components of the union has a diameter less than $1$, there is no pair of points \highlighttt{with integral distance} inside the
  components. Now consider arbitrary points $a$ and $b$ from two different connected components (open balls). Let $\alpha$ stand
  for the distance between their centers. The triangle inequality yields
  $$
    \alpha-\left(\frac{1-4\varepsilon}{2}\right)<\dist(a,b)<\alpha+\left(\frac{1-4\varepsilon}{2}\right).
  $$
  Since all possible distances $\alpha$ are given by $2k\sin\Bigl(\frac{j\pi}{p}\Bigr)$ for $1\le j\le \frac{p-1}{2}$, we look
  for a solution of the system of inequalities
  $$
    \left[2k\cdot\sin\!\left(\frac{j\pi}{p}\right)-\frac{1}{2}+2\varepsilon\right]\le 4\varepsilon
  $$
  where $k\in\mathbb{N}$. By Lemma~\ref{lemma_diagonals_p_gon} the factors $2\sin\!\left(\frac{j\pi}{p}\right)$ are
  irrational and linearly independent over $\mathbb{Q}$, so by Weyl's Theorem \cite{46.0278.06} such
  systems admit solutions for all \highlighttt{$\varepsilon$}.
  
  Therefore, for every $0<\varepsilon<\frac{1}{4}$ we can choose a suitable value of $k$ and construct an open $n$-component set without
  pairs of points an integral distance apart
  with volume $n\left(\frac{1}{2}-2\varepsilon\right)^d\cdot\lambda_d\!\left(B_d\right)$. As $\varepsilon$ approaches 0,
  this volume tends to $\frac{n}{2^d}\cdot\lambda_d\!\left(B_d\right)$, completing the proof. 
\end{proof}

Thus, in the case of round connected components the values of $l_d^\circ(n)$ and $f_d^\circ(n)$ are completely determined. In the general case of arbitrary connected components the problem is more challenging for $n\ge 2$ and will be addressed in the following section.

\section{Bounds for $\mathbf{l_d(n)}$ and \highlightt{the exact value of} $\mathbf{f_d(n)}$}
\label{sec_bounds}

In dimension \highlighttt{$1$} we can consider the disjoint union $U$ of $n$ open intervals of length $\frac{1}{n}$ inside an open interval of length 1.
Obviously $U$ avoids integral distances and the line intersection property holds trivially for $U$ the total length of $U$ being 1 which is the largest possible by Theorem 1. Thus  $f_1(n)=l_1(n)=1$  for all $n$. 
For $n=1$, we have $f_d(1)=l_d(1)=\lambda_d\!\left(B_d\right)$ and only $d$-dimensional open balls of diameter \highlighttt{$1$} can have that large volume. For $n,d\neq 1$
the evaluation of $f_d(n)$ and $l_d(n)$ gets more involved. In Subsection~\ref{subsec_two_components} we treat the 2-component case $n=2$. As to the general case, we only could find some bounds for $l_d(n)$ in  Subsection~\ref{subsec_bounds_line_restricted}
and succeeded in determining the exact values of the function $f_d(n)$ in Subsection~\ref{subsec_bounds_anti_integral}.  

\subsection{Two components}
\label{subsec_two_components}

At first we find an upper bound for $f_d(2)$. To this end, note that the condition in Lemma~\ref{lemma_prestage_width} can be restated as follows:
$\diam(\mathcal{C}_1\cup \mathcal{C}_2)\le \left\lfloor\dist(\mathcal{C}_1,\mathcal{C}_2)\right\rfloor+1$. We further use
lemmas~\ref{lemma_prestage_width} and \ref{lemma_repultion_property}
to provide a structural property of the pairs of connected components $\mathcal{C}_1$, $\mathcal{C}_2$ of a $d$-dimensional open set $\mathcal{P}$ avoiding integral distances. By  Lemma~\ref{lemma_repultion_property} there exist
parallel hyperplanes $\mathcal{H}_2$ and $\mathcal{H}_3$ such that, possibly after  relabeling the components,
$\mathcal{C}_1$ is on the left hand side of $\mathcal{H}_2$, $\mathcal{C}_2$ is on the right hand side of $\mathcal{H}_3$,
and $\mathcal{H}_2$ is on the left hand side of $\mathcal{H}_3$. W.l.o.g.\ we can assume that $\dist(\mathcal{H}_2,\mathcal{H}_3)\le
\dist(\mathcal{C}_1,\mathcal{C}_2)$. By Lemma~\ref{lemma_prestage_width} there exist another pair of hyperplanes $\mathcal{H}_1$,
$\mathcal{H}_4$ parallel to $\mathcal{H}_2$ and $\mathcal{H}_3$ such that $\mathcal{C}_1$ is on the right hand side of
$\mathcal{H}_1$ and $\mathcal{C}_2$ is on the left hand side of $\mathcal{H}_4$, that is, $\mathcal{C}_1$ lies between
$\mathcal{H}_1$ and $\mathcal{H}_2$, $\mathcal{C}_2$ lies between $\mathcal{H}_3$ and $\mathcal{H}_4$. W.l.o.g.\ we can
assume that $\dist(\mathcal{H}_1,\mathcal{H}_4)\le\diam(\mathcal{C}_1\cup \mathcal{C}_2)$. Thus for $d_1:=\dist(\mathcal{H}_1,
\mathcal{H}_2)$ and $d_2:=\dist(\mathcal{H}_3,\mathcal{H}_4)$ we have $d_1+d_2\le 1$ by Theorem~\ref{thm_line_intersection}(i).
Clearly $d_1$ and $d_2$ are upper bounds for the widths of $\mathcal{C}_1$ and $\mathcal{C}_2$ respectively.

For a convex body $\mathcal{K}$ in $\mathbb{R}^d$ with diameter $D$ and minimum width $\omega$ an upper bound for its $d$-dimensional
volume $V$ has been found in \cite[Theorem 1]{1074.52004}, namely:
\begin{equation}
  \label{ie_isoperimetric_width}
  V\le \lambda_{d-1}(B_{d-1})\cdot D^d\int_0^{\arcsin \frac{\omega}{D}}\cos ^d\theta\,\mbox{d}\theta.
\end{equation}
Equality holds if and only if $\mathcal{K}$ is the $d$-dimensional spherical symmetric slice with diameter~$D$ and minimum
width~$\omega$. In the planar case some more inequalities relating several descriptive parameters of a convex set can be
found in \cite{0955.52007}. Since we will extensively use $d$-dimensional spherical symmetric slices with diameter $1$ and
width $\frac{1}{2}$, we denote them by $S_d$. Viewing $S_d$ as a truncated $d$-dimensional open ball of unit diameter we denote
the two congruent cut-off bodies by $C_d$ and call them {\it caps}. Thus $\lambda_d(B_d)=\lambda_d(S_d)+2\cdot\lambda_d(C_d)$ where
\begin{eqnarray}
  \lambda_d(S_d)&=& \lambda_{d-1}(B_{d-1})\int\limits_0^{\frac{\pi}{6}}\cos ^d\theta\,\mbox{d}\theta,\label{eq_S_d}\\
  \lambda_d(C_d)&=& \frac{1}{2}\cdot\left(\lambda_d(B_d)-\lambda_d(S_d)\right)\label{eq_C_d}.
\end{eqnarray}
In Table~\ref{table_values_slices_and_caps} we tabulated the first few exact values of the  volumes of $S_d$ and $C_d$ and refer to the Appendix, \highlight{ 
Subsection~\ref{subsec_volume_formulas},} for more information on these volumes as functions of $d$.

\begin{table}[htp]
  \begin{center}
    \begin{tabular}{lllll}
    \toprule
    $\mathbf{d}$              & 2                                                  & 3                                & 4 & 5\\ 
    $\mathbf{\lambda_d(S_d)}$ & $\frac{\sqrt{3}}{8}+\frac{\pi}{12}\approx 0.4783$  & $\frac{11\pi}{96}\approx 0.3600$ &
    $\frac{\pi}{384}\cdot\left(9\sqrt{3}+4\pi\right)\approx 0.2303$ & $\frac{203\pi^2}{15360}\approx0.1304$\\[1mm]
    $\mathbf{\lambda_d(C_d)}$ & $\frac{\pi}{12}-\frac{\sqrt{3}}{16}\approx 0.1535$ & $\frac{5\pi}{192}\approx 0.0818$ & 
    $\frac{\pi^2}{96}-\frac{3\sqrt{3}\pi}{256}\approx0.0390$& $\frac{53\pi^2}{30720}\approx0.0170$\\[1mm]
    \bottomrule
    \end{tabular}
    \caption{Values of $\lambda_d(S_d)$ and $\lambda_d(C_d)$ for small dimensions.}
    \label{table_values_slices_and_caps}
  \end{center}
\end{table}

\begin{lemma}
  \label{lemma_upper_bound_general_integral_n_2}
   $f_d(2)\le 2\lambda_d(S_d)$ for all $d\ge 2$.
\end{lemma}
\begin{proof}
  With notation above we estimate making use of Inequality~(\ref{ie_isoperimetric_width}) the total volume of the closed convex hulls of the two connected
  components $\conv(\overline{\mathcal{C}}_1),\conv(\overline{\mathcal{C}}_2)$, i.e.
  $$
    \lambda_d(\conv(\mathcal{C}_1))+\lambda_d(\conv(\mathcal{C}_2))
  $$
  where both connected components are of diameter at most~$1$, $\mathcal{C}_1$ is of width  at most $d_1$, and $\mathcal{C}_2$
  is of width at most $d_2$. Thus we have
  $$
    \lambda_d(\conv(\mathcal{C}_1))\le \lambda_d(B_{d-1})\int_0^{\arcsin d_1}\cos ^d\theta\,\mbox{d}\theta
  $$
  and
  $$
    \lambda_d(\conv(\mathcal{C}_2))\le \lambda_d(B_{d-1})\int_0^{\arcsin d_2}\cos ^d\theta\,\mbox{d}\theta.
  $$
  Since both right hand sides are strictly monotone in $d_1$, $d_2$ respectively, we can assume w.l.o.g.\ that $d_1+d_2=1$, so 
  it suffices to maximize the following function of $x$
  $$
    \int_0^{\arcsin x}\cos ^d\theta\,\mbox{d}\theta+\int_0^{\arcsin (1-x)}\cos ^d\theta\,\mbox{d}\theta
  $$
  with domain $[0,1]$. A straightforward calculation shows that the function attains its unique maximum value at $x=\frac{1}{2}$.
\end{proof}

\begin{lemma}
  \label{lemma_two_component_construction}
  $$f_d(2)\ge 2\cdot \lambda_d(S_d)\quad\mbox{for all $d$} $$
\end{lemma}
\begin{proof}
  For an arbitrary integer $k\ge 5$ we place the center of a $d$-dimensional open ball with diameter $1-\frac{2}{k}$ at the origin
  and cut off the spherical caps with the hyperplanes determined by the values $\pm\left(\frac{1}{4}-\frac{1}{k}\right)$ of the first
  coordinate.  We denote by $\mathcal{S}_1$ the \highlight{resulting} truncated ball. We consider the copy $S_2$ of $S_1$ by shifting
  the center of $S_1$ $dk+\frac{1}{2}-\frac{2}{k}$ units along the first coordinate axis (Figure~\ref{fig_two_component_construction}
  below illustrates the 2-dimensional case). Since both $\mathcal{S}_1$ and $\mathcal{S}_2$ have diameter less than $1$ for all
  $k\in\mathbb{N}$, they contain no pair of points with integral distance. For arbitrary points $a\in\mathcal{S}_1$ and $b\in\mathcal{S}_2$, we have
  $$
    dk<\dist(a,b)<\sqrt{(d-1)\left(1-\frac{2}{k}\right)^2+\left(dk+1-\frac{4}{k}\right)^2}\le dk+1,
  $$
  so  $\mathcal{S}_1\cup\mathcal{S}_2$ has no pairs of points with integral distance.
  
  It is easily seen that the volume of $\mathcal{S}_1\cup\mathcal{S}_2$ approaches $2\cdot\lambda_d(S_d)$, as $k$ increases.
\end{proof}

\begin{figure}[htp]
  \begin{center}
    \includegraphics{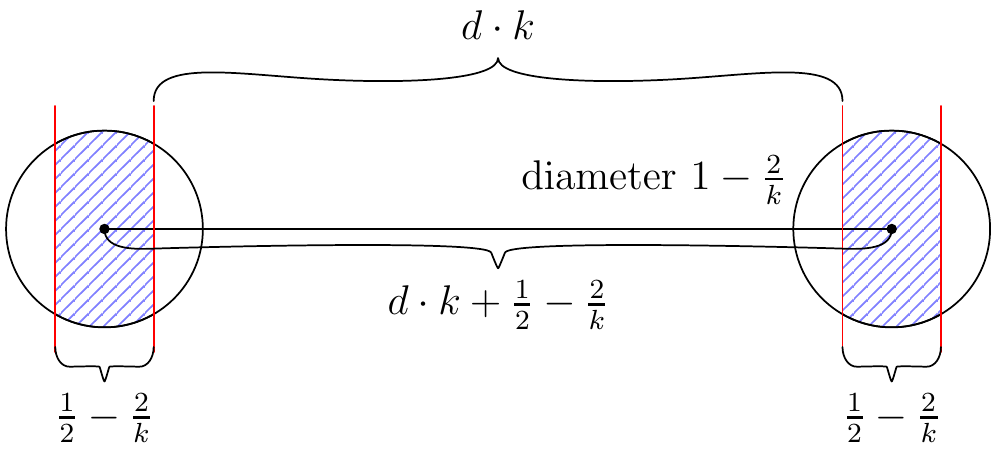}
    \caption{Truncated circles -- a construction of two components without integral distances.}
    \label{fig_two_component_construction}
  \end{center}
\end{figure}

Combining Lemmas~\ref{lemma_upper_bound_general_integral_n_2} and 4 %Lemma~\ref{lemma_two_component_construction} 
yields the following
\begin{corollary}
    $f_d(2)=2\lambda_d(S_d)$\,\, for all $d\ge 2$.
\end{corollary}

One might conjecture that \highlight{the upper bound  from Lemma~\ref{lemma_upper_bound_general_integral_n_2}} is also
valid for $l_d(2)$, see Conjecture~\ref{conj_l_d_n}. \highlight{Technically, we have used Lemmas~\ref{lemma_prestage_width} and 4
%Lemma~\ref{lemma_repultion_property}, 
but it is conceivable that there is an alternative approach not relying on these assertions.}

 Note that related problems can be quite complicated, e.g.\ it is hard
to determine the equilateral $n$-gon with diameter $1$ and maximum area \cite{1179.90306,1182.65088}. 

\begin{conjecture}
  \label{conj_l_d_n}
   $l_d(n)=n\cdot \lambda_d(S_d)$ for all $n\ge 2$ and $d\ge 2$.
\end{conjecture}

\subsection{Bounds for $\mathbf{l_d(n)}$}%the volume of point sets with restrictions on the lengths of line intersections}
\label{subsec_bounds_line_restricted}

Using exhaustion over lines, we can find two first upper bounds for $l_d(n)$.

\begin{lemma}
  \label{upper_bound_two_components_d}
  $l_d(2)\le \lambda_{d-1}\!\left(B_{d-1}\right)\cdot \left(\sqrt{\frac{2d}{d+1}}\right)^{d-1}$\,\, for all $d\ge 2$.
\end{lemma}
\begin{proof}
  By Lemma~\ref{lemma_diameter} both connected components, denoted by $\mathcal{C}_1$ and $\mathcal{C}_2$, are of diameter at
  most \highlighttt{$1$}, so Jung's theorem \cite{MR1389515,32.0296.05} yields the enclosing balls $\mathcal{B}_1$, $\mathcal{B}_2$ for these connected
  components of diameter $\sqrt{\frac{2d}{d+1}}$. 
  So there is an enclosing cylinder, having a $(d-1)$-dimensional ball of diameter  $\sqrt{\frac{2d}{d+1}}$ as its
  base, containing the closed convex hull $\conv(\overline{\mathcal{B}_1\cup\mathcal{B}_2})$. The diagram is depicted in
  Figure~\ref{fig_enclosing_tube_and_balls}, note that in general the two enclosing balls $\mathcal{B}_1$ and $\mathcal{B}_2$ are not necessarily disjoint. By exhausting the cylinder with the lines parallel to the line through the
  centers of $\mathcal{B}_1$ and $\mathcal{B}_2$ and applying
  Theorem~\ref{thm_line_intersection}(i)  we conclude, using a suitable Riemann integral or Fubini's theorem, that the volume of $\mathcal{C}_1\cup\mathcal{C}_2$ is at most
  $\lambda_{d-1}\!\left(B_{d-1}\right)\cdot \left(\sqrt{\frac{2d}{d+1}}\right)^{d-1}$.
\end{proof}

\begin{figure}[htp]
  \begin{center}
    \includegraphics[width=12cm]{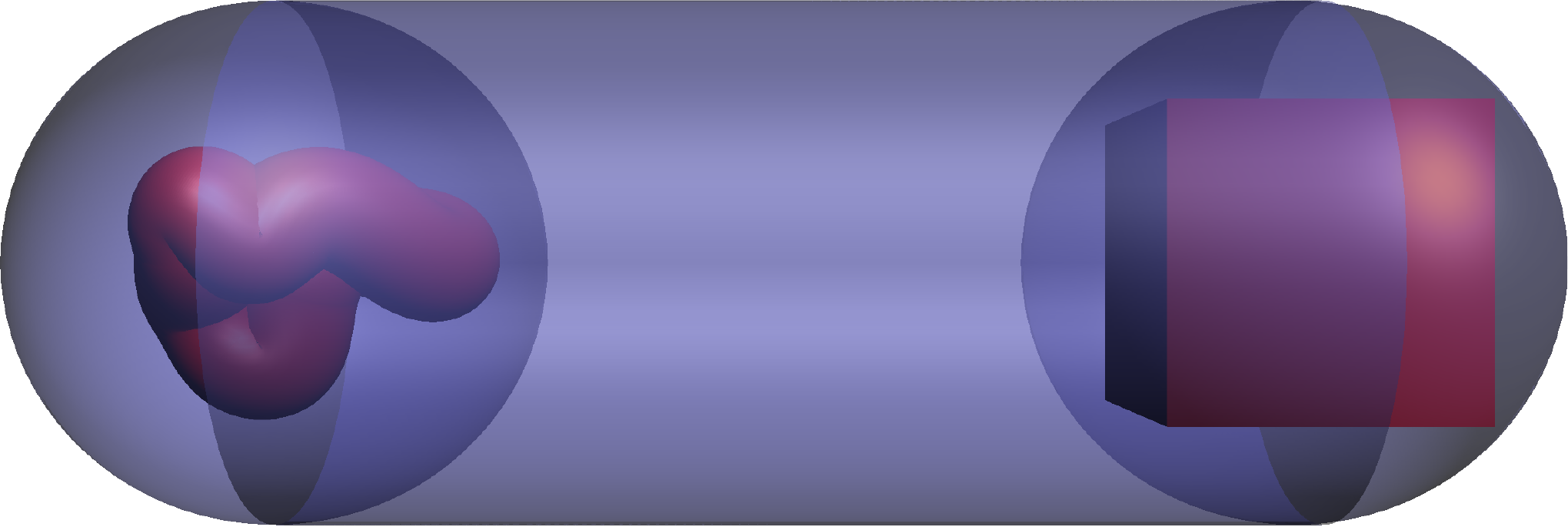}
    \caption{Two $3$-dimensional components with the enclosing balls and enclosing cylinder.}
    \label{fig_enclosing_tube_and_balls}
  \end{center}
\end{figure}

The estimates for the first few upper bounds of $l_d(2)$ in Lemma~\ref{upper_bound_two_components_d} are:
$l_2(2)\le \frac{2}{\sqrt{3}}\approx 1.1547$, $l_3(2)\le \frac{3\pi}{8}\approx 1.1781$, $l_4(2)\le \frac{8\sqrt{2}\pi}{15\sqrt{5}}\approx 1.0597$,
$l_5(2)\le \frac{25\pi^2}{288}\approx 0.8567$ and $l_d(2)$ tends to 0 as the dimension~$d$ increases.

Note that we used a bit wastefully the Jung enclosing balls. The universal cover problem, first stated in a personal communication of Lebesgue in 1914, asks
for the minimum area $A$ of a convex set $U$ containing a congruent copy of any planar set of diameter $1$, see \cite{1088.52002}. For the currently
best known bounds $0.832\le A\le 0.844$ and generalizations to higher dimensions we refer the interested reader to \cite[Section 11.4]{1086.52001}. 
\highlight{In this paper we} do not \highlight{pursue the aim of finding more precise bounds  for the maximum volumes using this idea}. \highlight{The restriction of the shape of connected components to $d$-dimensional open balls 
has already been treated in Section~\ref{sec_collection_of_balls}.}

\begin{figure}[htp]
  \begin{center}
    \includegraphics[width=4.5cm]{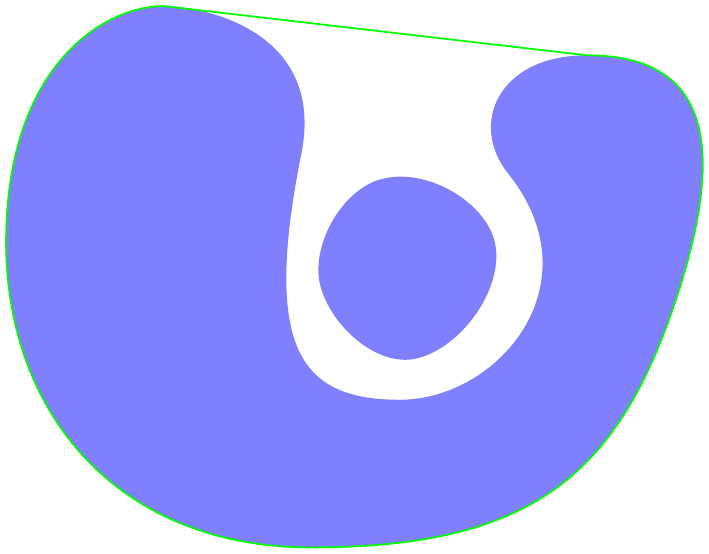}
    \caption{A connected component contained in the convex hull of another one.}
    \label{fig_contained_components}
  \end{center}
\end{figure}

In dimension $d=2$ the upper bound from Lemma~\ref{upper_bound_two_components_d} can easily be improved.
\begin{lemma}
  \label{lemma_upper_bound_2_2}
  $$l_2(2)\le 1.$$
\end{lemma}
\begin{proof}
  Let $\mathcal{P}$ be a planar open point set with two connected components $\mathcal{C}_1$ and $\mathcal{C}_2$ of
  diameter at most $1$ each. If one of them is contained in the closed convex hull of the other, see Figure~\ref{fig_contained_components}
  for an example, then we have $\lambda_2(\mathcal{P})\le \lambda_2(B_2)=\frac{\pi}{4}<1$. Otherwise, we select any support
  line $\mathcal{L}$ through the boundary points of $\mathcal{C}_1$ and $\mathcal{C}_2$ so that both regions are in the same
  half-plane determined by $\mathcal{L}$. We then consider the strip parallel to this line with smallest possible width~$w$
  containing both regions, see Figure~\ref{fig_two_components_between_parallel_lines}. Since both $\mathcal{C}_1$ and 
  $\mathcal{C}_2$ have diameter at most $1$, we have $w\le 1$. By exhausting the strip with the lines parallel to $\mathcal{L}$
  and applying Theorem~\ref{thm_line_intersection}(i) we conclude, using Riemann integral or Steiner symmetrization with respect to a line orthogonal to $\mathcal{L}$, that the area of $\mathcal{C}_1\cup\mathcal{C}_2$ is at most $1$.
\end{proof}

\begin{figure}[htp]
  \begin{center}
    \includegraphics{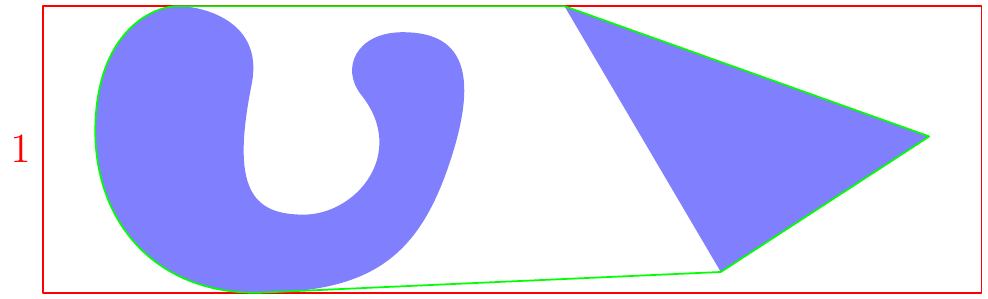}
    \caption{Two components between two parallel lines.}
    \label{fig_two_components_between_parallel_lines}
  \end{center}
\end{figure}

\subsection{\highlightt{The exact value of $\mathbf{f_d(n)}$}}
\label{subsec_bounds_anti_integral}

\highlight{Combining Lemmas~\ref{lemma_upper_bound_general_integral_n_2} and 5 %Lemma~\ref{lemma_averaging}
 yields
the upper bound $f_d(n)\le n\lambda_d\!\left(S_d\right)$.} \highlightt{In the remaining part of this subsection we will
describe configurations whose volumes asymptotically attain this upper bound.}

\begin{figure}[htp]
  \begin{center}
    \includegraphics[width=8cm]{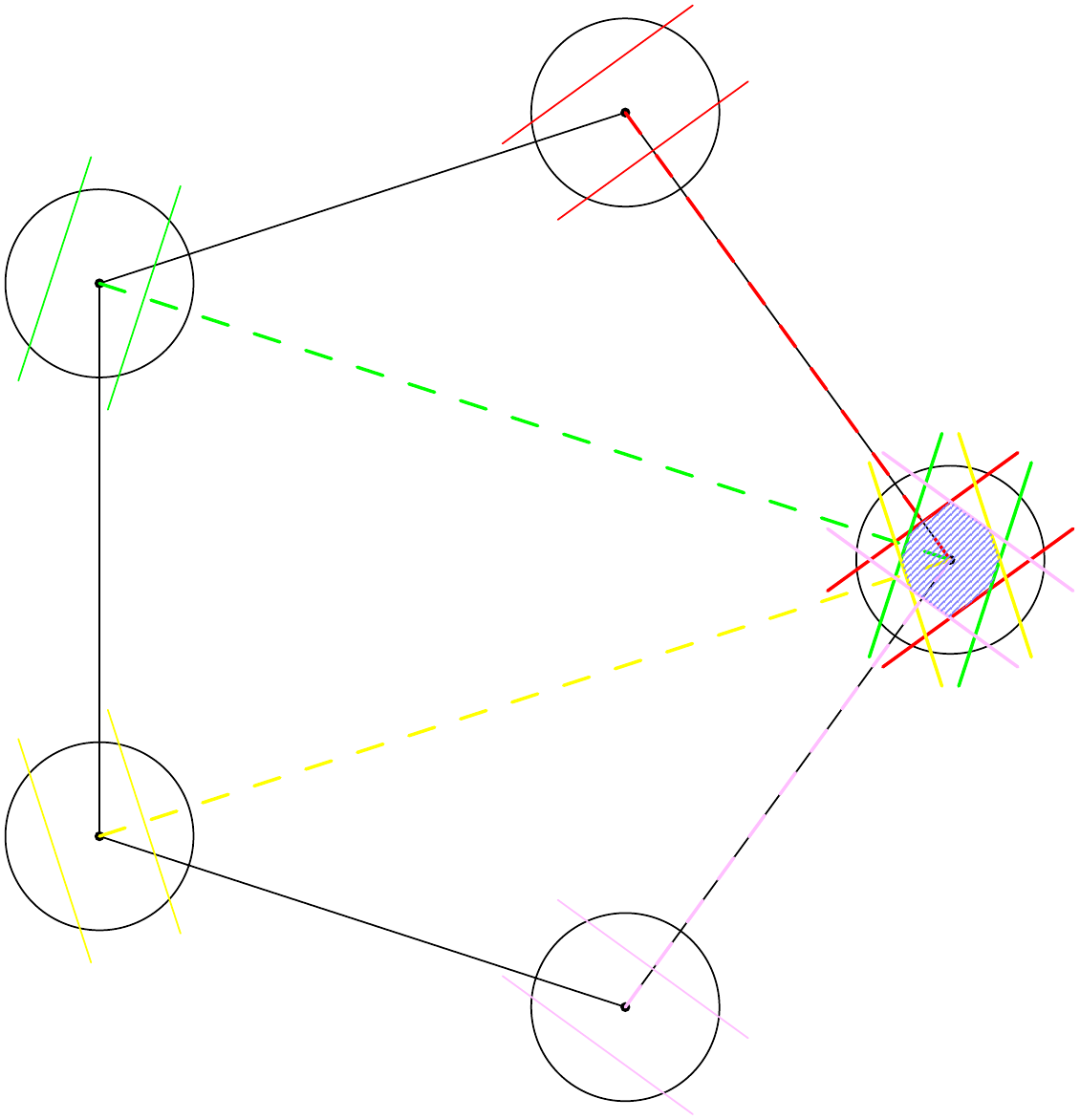}
    \caption{$p$-gon construction: open set avoiding integral distances for $d=2$ and $p=n=5$.}
    \label{fig_pentagon_construction}
  \end{center}
\end{figure}

\highlightt{As a first step, } we improve slightly the construction from Theorem~\ref{construction_f_circ_d_n}.
For $d\ge 2$, we choose an odd prime $p\ge n$ and locate the centers of $n$ open balls of diameter $1-2\varepsilon$, where $\varepsilon$ is suitably chosen, at $n$ consequtive vertices of a regular $p$-gon. For each two balls, we cut off spherical
caps in the directions of the lines through their centers the resulting sets being of width  $\frac{1}{2}-2\varepsilon$. We can assume that $\varepsilon$ approaches 0, as the
circumradius of the $p$-gon increases. For our purpose it suffices to consider a regular $p$-gon $P$ of fixed circumradius $>2$, locate the centers of $n$ open balls at the consequtive vertices of $P$, and cut off spherical caps so that the connected components of the resulting union  are of width 
$\frac{1}{2}$ in the direction of each line through the corresponding vertices, i.e.\ the centers of the $n$ balls.
For future reference we call this construction a \textit{$p$-gon construction}. An example of such a construction for $p=n=5$ in
dimension $d=2$ is depicted in \highlightt{Figure}~\ref{fig_pentagon_construction}. 

\begin{theorem}
    $f_d(n)=n\lambda_d(S_d)$\,\, for all $d\ge 2$ and $n\ge 2$.
\end{theorem}
\begin{proof}
  \highlighttt{
  It follows from Lemmas~\ref{lemma_upper_bound_general_integral_n_2} and 5 that %\ref{lemma_averaging} 
  $f_d(n)\le n\lambda_d(S_d)$. By 
  Lemma~\ref{lemma_two_component_construction} we can assume that $n\ge 3$. For arbitrary $\varepsilon$ we denote by $S_{d,\varepsilon}$
  a $d$-dimensional spherical symmetrical slice with diameter $1-2\varepsilon$ and minimum width $\frac{1}{2}-2\varepsilon$.
  As $\varepsilon$ approaches $0$, the volume of $S_{d,\varepsilon}$ tends to $\lambda_d(S_d)$. Below we provide a construction
  of an open $n$-component point set $\mathcal{P}'$ avoiding integral distances each of whose connected components contains a congruent
  copy of $S_{d,\varepsilon}$.}
  
  \highlighttt{Consider a regular $p$-gon $P$ with circumradius $k$, the parameters $p$ and $k$ are to be specified. We enumerate clockwise the
  vertices of $P$ from $1$ to $p$ and assume w.l.o.g.\ that the line through the vertices $1$ and $2$ is the $x$-axis. At each vertex $1\le i\le n\le p$
  we place the center of an open $d$-dimensional ball of diameter $1-\varepsilon$.
  For each pair of the $n$ balls we cut off spherical caps in the direction of the lines through their centers resulting
  in a set of width $\frac{1}{2}-\varepsilon$. We denote the union of the resulting $n$ open sets by $\mathcal{P}$.}
  
  \highlighttt{Consider further all  $2\cdot {n \choose 2}$ cutting hyperplanes that cut off the spherical caps from the initial open balls. As
  the number $p$ of vertices of the $p$-gon $P$ increases, with $n$ fixed, all those hyperplanes tend to be orthogonal to the $x$-axis. Now choose
  a prime $p$ large enough so that each connected component of $\mathcal{P}$ contains a $d$-dimensional spherical symmetrical 
  slice with diameter $1-2\varepsilon$ and minimal width $\frac{1}{2}-2\varepsilon$ whose cutting hyperplanes are orthogonal to the $x$-axis.
  By $\mathcal{P}'$ we denote the subset of $\mathcal{P}$ which is the union of those $S_{d,\varepsilon}$'s.}
  
  \highlighttt{There exists a number $k_1$ such that for $k\ge k_1$ each line hits at most two connected components of $\mathcal{P}'$. Since the diameter
  of each of its connected components is at most $1-2\varepsilon$, the pairwise distances between the points within the same component are non-integral.
  Let $a$ and $b$ be two points in different connected components. By the construction
  the distance between the corresponding centers is given by $2k\cdot\sin\!\left(\frac{j\pi}{p}\right)$ for a suitable integer $j$. Thus
  $$
    \dist(a,b)\ge 2k\cdot\sin\!\left(\frac{j\pi}{p}\right)-\frac{1}{2}+\varepsilon.
  $$}
  
  \highlighttt{There exists a number $k_2$ such that for $k\ge k_2$, we have
  $$
    \dist(a,b)\le 2k\cdot\sin\!\left(\frac{j\pi}{p}\right)+\frac{1}{2}-\varepsilon,
  $$
  since all the lines joining the centers of the connected components of $\mathcal{P}'$ tend to be parallel to the $x$-axis, as $k$ increases. (cf. the proof
  of Lemma~\ref{lemma_two_component_construction}.)}
  
  \highlighttt{Thus, provided that for $k\ge\max\{k_1,k_2\}$, the system of inequalities 
  $$
    \left[2k\cdot\sin\!\left(\frac{j\pi}{p}\right)-\frac{1}{2}+\varepsilon\right]\le 2\varepsilon
  $$
  has a solution, the distance $\dist(a,b)$ can not be integral, so $\mathcal{P}'$ does not contain a pair of points an integral distance apart.
  By Lemma~\ref{lemma_diagonals_p_gon} and the Weyl theorem the above system indeed admits a solution for all $k$}. This completes the proof.
\end{proof}

\section{Conclusion}
\label{sec_conclusion}
\noindent
Problems related to point sets with pairwise rational or integral distances were one of Erd\H{o}s' favorite subjects in
combinatorial geometry. In the present paper we study \highlighttt{a counterpart to this type of problems} by asking for the largest open
$d$-dimensional set $\mathcal{P}$ of points without a pair of points an integral distance apart, i.e.\ that with the largest
possible volume $f_d(n)$, where $n$ stands for the number of connected components of $\mathcal{P}$. As a relaxation we
have also considered $d$-dimensional open point sets with $n$ connected components of  diameter at most $1$ each whose intersection with every line has a total length of at most \highlighttt{$1$}. The corresponding \highlighttt{maximum}
volume has been denoted by $l_d(n)$. While the assumption on the diameters of the connected components seems to be a bit
technical, geometrical objects \highlighttt{with specified} intersections with lines or higher-dimensional subspaces are interesting in their own right.
In this context we just mention the famous Kakeya problem \highlighttt{of whether a Kakeya set in $\mathbb{R}^d$, i.e.\ a compact set
containing a unit line segment in every direction, has Hausdorff dimension $d$ , see e.g.\ the review \cite{wolff1999recent} or \cite[Problem G6]{UPIG}.}

By restricting the shapes of the connected components to $d$-dimensional open  balls, we were able to determine the exact
values of the corresponding \highlighttt{maximum} volumes $f^\circ_d(n)$ and $l^\circ_d(n)$ respectively. \highlightt{Also the values of
$f_d(n)$ have been determined exactly, while for $l_d(n)$ we only have the lower bound $l_d(n)\ge f_d(n)$, which we conjecture
to be tight.} 

\section*{Acknowledgements}
The authors thank  Juris Steprans, Peter Biryukov, Robert Connelly, and Andrey Verevkin for valuable comments and discussions, Tobias Kreisel
for producing several graphics, Thomas Kriecherbauer and Benoit Cloitre for some remarks on integrals and integer sequences. We also thank \highlight{the anonymous referees for their extensive comments, which helped very much to improve the presentation
of this paper.}

\bibliography{anti_integral}
\bibdata{anti_integral}
\bibliographystyle{amsplain}  

\appendix
\section{Appendix}

\highlight{
In order to keep the main part of the paper more accessible we have moved some side remarks and necessary technical computations to
this Appendix.} 

\subsection{\highlight{Details of the annuli construction}}
\label{subsec_annuli}
\noindent
\highlight{
%\textsc{Proof of Construction~\ref{annuli_construction}.}
We shall show that Example~\ref{annuli_construction} satisfies the properties as stated.
First note that both $\mathcal{A}_n^d$ and $\mathcal{A}_{n+1}^d$ meet $\mathcal{B}_n^d$ for $n\ge 1$. Thus $\mathcal{P}$
is a connected open set in $\mathbb{R}^d$. The volume $\lambda_d\left(\mathcal{A}_n^d\right)$ is given by
$$
  \lambda_d(B_d)\cdot\left(\left(2n+\frac{2}{dn^d}\right)^d-(2n)^d\right)=\lambda_d(B_d)\cdot2^d\cdot\left(\left(n+\frac{1}
  {dn^d}\right)^d-n^d\right)\ge\lambda_d(B_d)\cdot2^d\cdot\frac{1}{n}.
$$
Since the harmonic series diverges to infinity, the $d$-dimensional volume of $\mathcal{P}$ is unbounded.
}

\highlight{
Now we consider the intersection of a line $\mathcal{L}$ with a $d$-dimensional annulus $\mathcal{C}_d(r_1,r_2)$ of inner
radius $r_1$ and of outer radius $r_2$ centered at the origin. By symmetry we can assume that $\mathcal{L}$ is parallel to
the $x$-axis, i.e.\ $\mathcal{L}=\left\{\begin{pmatrix}1&0&\dots&0\end{pmatrix}^T\cdot \lambda+ \begin{pmatrix}0&a_2&\dots&
a_d\end{pmatrix}^T\mid\lambda\in\mathbb{R}\right\}$. Furthermore, we can also assume by symmetry that $a_i\ge 0$ for all $2\le i\le d$.
\highlight{To simplify notation we set} $l:=\sqrt{\sum_{i=2}^d a_i^2}$. Note that $\mathcal{C}_d(r_1,r_2)\cap\mathcal{L}=\emptyset$ for $l^2>r_2^2$. The $x$-coordinates of the intersections of $\mathcal{L}$ with the $d$-dimensional sphere of radius $r_1$ are given by
$\pm \sqrt{r_1^2-l^2}$, as long as $l^2\le r_1^2$. Similarly the $x$-coordinates of the intersections of $\mathcal{L}$ and the $d$-dimensional sphere
of radius $r_2$ are given by $\pm \sqrt{r_2^2-l^2}$, as long as $l^2\le r_2^2$. For $l^2\le r_1^2$, we have
$$
  \lambda_1\left(\mathcal{C}_d(r_1,r_2)\cap\mathcal{L}\right)=2\cdot\underset{=:h_1(a_2,\dots,a_d)}{
  \underbrace{\left(\sqrt{r_2^2-\sum\limits_{i=2}^d a_i^2}-\sqrt{r_1^2-\sum\limits_{i=2}^d a_i^2}\right)}}.
$$
Since $$\frac{\partial h_1}{\partial a_i}(a_2,\dots,a_d)=a_i\cdot\left(\frac{1}{\sqrt{r_1^2-\sum\limits_{i=2}^d a_i^2}}-
\frac{1}{\sqrt{r_2^2-\sum\limits_{i=2}^d a_i^2}}\right)\ge0$$, we can assume $l^2\ge r_1^2$ for the
maximum length of the line intersection. If the $a_i$ are restricted by an inequality $l^2\le k^2\le r_1^2$, the maximum length of the intersection
is bounded above by $2\sqrt{r_2^2-k^2}-2\sqrt{r_1^2-k^2}$.
}

\highlight{
For $r_1^2\le\sum_{i=2}^d a_i^2\le r_2^2$, we have
$$
  \lambda_1\left(\mathcal{C}_d(r_1,r_2)\cap\mathcal{L}\right)=2\cdot\underset{=:h_2(a_2,\dots,a_d)}{
  \underbrace{\sqrt{r_2^2-\sum\limits_{i=2}^d a_i^2}}}
$$
and
$$
  \frac{\partial h_1}{\partial a_2}(a_2,\dots,a_d)=-a_i\cdot \frac{1}{\sqrt{r_2^2-\sum\limits_{i=2}^d a_i^2}}\le 0,
$$
so the extreme values \highlight{are} attained at $\sum_{i=2}^d a_i^2=r_1^2$ where we have 
$\lambda_1\left(\mathcal{C}_d(r_1,r_2)\cap\mathcal{L}\right)\le 2\sqrt{r_2^2-r_1^2}$.
}

\highlight{
Thus for an arbitrary line $\mathcal{L}$, we have
$$
  \lambda_1\left(\cup_{n\ge30}\mathcal{B}_n^d\cap\mathcal{L}\right)\le \sum_{n=30}^\infty 
  2\sqrt{\left(1+\frac{1}{n^4}\right)^2-1^2}\le\sum_{n=30}^\infty\frac{2\sqrt{3}}{n^2}<0.12.
$$
For the remaining part we restrict ourselves with lines parallel to the $x$-axis. If $l<30$, then
\begin{eqnarray*}
  \lambda_1\left(\cup_{n\ge30}\mathcal{A}_n^d\cap\mathcal{L}\right)&\le& 2\sqrt{\left(30+\frac{1}{d\cdot 30^d}\right)^2-30^2}+
  \sum_{n=31}^\infty
  2\sqrt{\left(n+\frac{1}{dn^d}\right)^2-l^2}-2\sqrt{n^2-l^2}\\
  &\le& 0.366+2\sum\limits_{n=31}^\infty \frac{\frac{2}{n}}{2\sqrt{n^2-30^2}}<0.47.
\end{eqnarray*}
For $l\ge 30$, we have
\begin{eqnarray*}
  \lambda_1\left(\cup_{n\ge30}\mathcal{A}_n^d\cap\mathcal{L}\right)&\le&
  4\sqrt{\left(\lfloor l\rfloor +\frac{1}{d\cdot \lfloor l\rfloor^d}\right)^2-\lfloor l\rfloor^2}+
  \sum\limits_{n=\lfloor l+2\rfloor}^\infty 2\sqrt{\left(n+\frac{1}{dn^d}\right)^2-l^2}-2\sqrt{n^2-l^2}\\
  &\le& 0.732++2\int_{\lfloor l+1\rfloor}^\infty \frac{1}{x\sqrt{x^2-l^2}}\operatorname{d}x
  = 0.732+\frac{2}{l}\cdot\arcsin\left(\frac{l}{\lfloor l+1\rfloor}\right)\\
  &\le& 0.732 +\frac{2}{l} \cdot\frac{\pi}{2}<0.84.
\end{eqnarray*}  
Since $0.12+\max\{0.47,0.84\}<1$, we have $\lambda_1(\mathcal{P}\cap\mathcal{L})<1$ for all lines $\mathcal{L}$.%\hfill{$\square$}
}

\subsection{\highlight{Volumes of truncated balls and caps}}
\label{subsec_volume_formulas}

\highlight{In Table~\ref{table_values_slices_and_caps} we presented the volumes of truncated $d$-dimensional open balls of unit
 diameter $S_d$ and the cut-off bodies, i.e.\ caps $C_d$, in small dimensions $d$.  Equations (\ref{eq_S_d}) and 
(\ref{eq_C_d}) enable us to compute the values $v(d):=\int\limits_{0}^{\frac{\pi}{6}}\cos^d(x)\,\operatorname{d}x$. First
few values are given by $v(1)  = \frac{1}{2}$, $v(2)  = \frac{1}{8}\cdot\sqrt{3}+\frac{1}{12}\cdot \pi$, $v(3)  = \frac{11}{24}$, and
$v(4)  = \frac{9}{64}\cdot\sqrt{3}+\frac{1}{16}\cdot\pi$. Integrating by parts we find 
$$
  v(d)=\left\{\begin{array}{rl}
   \frac{(2m-1)!!}{(2m)!!}\cdot\left(\frac{1}{2}\cdot\sum\limits_{k=0}^{m-1}\frac{(2k)!!}{(2k+1)!!}\cdot \frac{\sqrt{3}}{2}\cdot\left(\frac{3}{4}\right)^{k}
   +\frac{\pi}{6}\right)&\text{for } d=2m,\\
   \frac{(2m)!!}{(2m+1)!!}\cdot\frac{1}{2}\cdot\sum\limits_{k=0}^m \frac{(2k-1)!!}{(2k)!!}\cdot \left(\frac{3}{4}\right)^{k}
   &\text{for } d=2m+1.
  \end{array}\right.
$$
Given the integer sequence A091814 from the ``On-line encyclopedia of integer sequences'', $v(d)$ can be written as $\frac{A091814(d))\cdot\left(\frac{d-1}{2}\right)!}{d!\cdot 2^{\frac{d+1}{2}}}$ for all odd $d$. Benoit Cloitre contributed the following second order recursion formula in this case: $v(1)=\frac{1}{2}$, $v(3)=\frac{11}{24}$, and $$v(2n-1)=\frac{1}{8n-4}\cdot\Bigl((14n-17)\cdot v(2n-3)-6(n-2)\cdot v(2n-5)\Bigr)$$ for $n\ge 3$. A similar recursion formula can be obtained for all even $d$, where
$v(d)$ can be written in the form  $q(d)\cdot\sqrt{3}+\frac{{{d-1}\choose {\frac{d}{2}}}}{2^d\cdot 3}\cdot\pi$ for some rational number $q(d)$.}

\highlight{
To determine the asymptotic behavior of $v(d)$ as $n\to\infty$ one can compute the corresponding ordinary generating function:
$$
  F(z):=\sum\limits_{k=0}^\infty v(k)z^k=\sum\limits_{k=0}^\infty\int\limits_0^{\frac{\pi}{6}} (z\cos t)^k\,\operatorname{d}t=
  \int\limits_0^{\frac{\pi}{6}}\frac{\operatorname{d}t}{1-z\cos t}=\frac{2}{\sqrt{1-z^2}}\arctan\!\left(\sqrt{\frac{1+z}{1-z}}\cdot\tan \frac{\pi}{12}\right).
$$
We apply the singularity analysis to determine the asymptotic behavior of $a_n:=F_\alpha(z)[z^n]$, where slightly more generally, $F_\alpha(z):= \frac{2}{\sqrt{1-z^2}}\arctan\!\left(\sqrt{\frac{1+z}{1-z}}\cdot\alpha\right)$, see e.g.\ \cite[chapter VI]{MR2483235}. The main singularity is at $z=1$, since there is a compensation for $z=-1$.
It follows from
\begin{eqnarray*}
  \arctan\!\left(\sqrt{\frac{1+z}{1-z}}\cdot\alpha\right) &=& \frac{\pi}{2}+O\bigl((1-z)^\frac{1}{2}\bigr),\\
  \frac{2}{\sqrt{1+z}} &=& \sqrt{2} +O(1-z),\text{ and}\\
  \bigl[z^n\bigr]\frac{1}{\sqrt{1-z}} &=& \frac{1}{\sqrt{\pi n}}+O\!\left(\frac{1}{n^{\frac{3}{2}}}\right) 
\end{eqnarray*}
that
$$
  a_n=\sqrt{\frac{\pi}{2n}}+O\!\left(\frac{1}{n^{\frac{3}{2}}}\right).
$$
Thus $v(d)\sim \sqrt{\frac{\pi}{2d}}$.}

\end{document}